\documentclass[hidelinks,onefignum,onetabnum]{siamart251216}

\usepackage{tikz}
\usepackage{subcaption}
\usepackage{booktabs}
\usepackage{amsfonts}
\usepackage{amssymb}
\usepackage{graphicx}
\usepackage{epstopdf}
\usepackage{algorithmic}
\ifpdf
  \DeclareGraphicsExtensions{.eps,.pdf,.png,.jpg}
\else
  \DeclareGraphicsExtensions{.eps}
\fi

\theoremstyle{plain}
\newtheorem{example}[theorem]{Example}

\newsiamremark{remark}{Remark}
\newsiamremark{hypothesis}{Hypothesis}
\crefname{hypothesis}{Hypothesis}{Hypotheses}
\newsiamthm{claim}{Claim}
\newsiamremark{fact}{Fact}
\crefname{fact}{Fact}{Facts}

\headers{Uniform Coefficientwise Backward Stability}{J. P\'erez}

\title{Scaled Companion Linearizations for Coefficientwise Backward Stability\thanks{Submitted to the editors DATE.
}
}

\author{Javier P\'erez\thanks{University of Montana, Missoula, MT 
  (\email{javier.perez-alvaro@mso.umt.edu}).}
}

\usepackage{amsopn}

\ifpdf
\hypersetup{
  pdftitle={Scaled companion linearizations for coefficientwise backward stability},
  pdfauthor={Javier P\'erez}
}
\fi

\begin{document}

\maketitle

\begin{abstract}
A backward stable eigensolver applied to a linearization of a matrix polynomial does not necessarily produce eigenpairs with small  coefficientwise relative backward errors. 
Scaling the polynomial can help, but a single scaling may not provide coefficientwise backward stability across the entire spectrum, particularly when the eigenvalues have widely varying magnitudes.

To address this difficulty, we use different scalings in different regions of the spectrum. 
We partition the nonzero complex plane into radial regions and, for each region, construct a scaled companion linearization with an eigenvector recovery rule. 
On each assigned region, we bound the polynomial backward error by the pencil backward error multiplied by an explicit degree-dependent constant and the ratio between the pencil-vector norm and the recovered-vector norm. 
We prove degree-dependent bounds for these ratios at exact pencil eigenvectors. 
When computed eigenvectors satisfy the same ratio bounds, the resulting backward-error amplification is \(O(d^{5/2})\), independently of coefficient norm disparities.
\end{abstract}



\section{Introduction}

Let
\begin{equation}\label{eq:matrix-polynomial}
    P(\lambda)=\sum_{k=0}^{d} P_k\lambda^k,
    \qquad P_k\in\mathbb{C}^{n\times n},
\end{equation}
be a \emph{regular matrix polynomial}, that is, assume that \(\det P(\lambda)\) is not identically zero. 
We call \(d\) the \emph{grade} of \(P\). 
The \emph{degree} of \(P\) is the largest index \(k\) for which \(P_k\ne 0\).

Throughout, we assume that $P$ has degree $d\geq2$ and that $P_0\neq0$.
Thus, both endpoint coefficients \(P_0\) and \(P_d\) are nonzero.

The \emph{polynomial eigenvalue problem} is to find $\lambda\in\mathbb C$ and nonzero vectors $x,y\in\mathbb C^n$ such that
\[
    P(\lambda)x=0,
    \qquad
    y^*P(\lambda)=0.
\]
We call $\lambda$ an \emph{eigenvalue} and $x$ and $y$ corresponding right and left \emph{eigenvectors}.
The pairs $(\lambda,x)$ and $(\lambda,y)$ are right and left \emph{eigenpairs}.
When there is no risk of confusion, we omit ``right.''

What should we expect from a numerical method for computing the eigenpairs of a matrix polynomial $P$?
Ideally, each computed eigenpair should be exact for a nearby polynomial
\[
    P+\Delta P
    =\sum_{k=0}^{d}(P_k+\Delta P_k)\lambda^k,
\]
with
\[
    \|\Delta P_k\|_2\leq C u\|P_k\|_2,
    \qquad 0\leq k\leq d,
\]
where $u$ is the \emph{unit roundoff} ($u=2^{-53}\approx 1.11\times 10^{-16}$ in IEEE double precision) and $C$ is a modest constant.
This measures each matrix coefficient perturbation relative to that coefficient's norm.
In particular, a zero coefficient must remain zero.
We refer to this requirement as  \emph{coefficientwise relative backward stability}.

A standard method for solving a polynomial eigenvalue problem is to replace it with a generalized eigenvalue problem
\[
    L(\lambda)v=0, \qquad w^*L(\lambda)=0,
    \qquad L(\lambda)=\lambda X+Y,
\]
where \(L\) is a larger \emph{matrix pencil} (a matrix polynomial of degree one) called a \emph{linearization of \(P\)}.
A linearization preserves the finite eigenvalues and their elementary divisors; a strong linearization also preserves those at infinity.
The QZ algorithm computes eigenvalues and eigenvectors of \(L\), from which eigenvectors of \(P\) can typically be recovered.

Many families of linearizations have been developed, including ansatz-space linearizations \cite{MackeyEtAl2006,NakatsukasaEtAl2017}, Fiedler pencils \cite{DeTeranEtAl2010}, and block minimal bases pencils \cite{BuenoEtAl2018,DopicoEtAl2018,DopicoPerezVanDooren2019}.
In this work, we focus on the classical Frobenius companion pencil.

The QZ algorithm is coefficientwise backward stable for the pencil \cite{MolerStewart1973}.
More precisely, QZ computes the exact generalized Schur decomposition of a nearby pencil $\lambda(X+\Delta X)+Y+\Delta Y$, where
\[
    \|\Delta X\|_2\le c_Nu\|X\|_2,
    \qquad
    \|\Delta Y\|_2\le c_Nu\|Y\|_2,
\]
where \(c_N\) depends on the pencil size \(N\).
However, small perturbations to the pencil coefficients need not correspond to small  perturbations to every matrix polynomial coefficient.
A computed pencil eigenpair can therefore have a small backward error while its recovered matrix polynomial eigenpair has a large one. 
Tisseur~\cite{Tisseur2000} and Higham, Li, and Tisseur~\cite{HighamLiTisseur2007} analyze how pencil backward errors relate to the backward errors of the recovered matrix polynomial eigenpairs.
Their analysis shows that coefficientwise backward stability of the pencil alone does not guarantee coefficientwise backward stability for $P$.

Scaling $P$ can reduce the amplification of pencil backward errors when recovering matrix polynomial eigenpairs.
A common approach is tropical scaling, which uses the roots of the \emph{tropical polynomial}
\[
    T_P(r)=\max_{0\leq k\leq d}\|P_k\|_2r^k
\]
to choose scaling parameters \cite{GaubertSharify2009,NoferiniSharifyTisseur2015}.

Scaling at a tropical root is particularly useful for eigenvalues whose moduli lie near that tropical root \cite{GaubertSharify2009,NoferiniSharifyTisseur2015}.
But how close are the eigenvalues to the tropical roots?
Noferini, Sharify, and Tisseur give localization bounds that depend on root separation and the conditioning of the relevant coefficient matrices \cite[Theorem~2.7 and Corollary~2.8]{NoferiniSharifyTisseur2015}.
Under suitable conditions, these bounds place groups of eigenvalues near the tropical roots.
Without those conditions, eigenvalues can lie far from every tropical root. 
Scaling at the tropical roots alone may therefore fail to guarantee small coefficientwise backward errors across the spectrum.

In contrast, when the tropical polynomial has only one distinct root, scaling at that root gives coefficientwise backward stability across the spectrum.
In the quadratic case, this is the regime addressed by \texttt{quadeig} of Hammarling, Munro, and Tisseur~\cite{HammarlingMunroTisseur2013}, which combines tropical scaling, a suitable linearization, and an eigenvector recovery rule. 
Its analysis establishes coefficientwise backward stability for quadratic matrix polynomials with only one tropical root (known as not-too-heavily-damped matrix polynomials).

For quadratic matrix polynomials with two distinct tropical roots (the heavily damped case), Zeng and Su~\cite{ZengSu2014} use two linearizations of a scaled polynomial: one for eigenvalues with modulus at most one and the other for eigenvalues with modulus at least one.
Their analysis guarantees coefficientwise backward stability for the recovered eigenpairs in both regions.

We extend the approach of Zeng and Su to matrix polynomials of arbitrary degree \(d\geq 2\).
Assume that \(\|P_0\|_2,\|P_d\|_2>0\), and let
\[
    \alpha_1<\cdots<\alpha_s
\]
be the distinct tropical roots of \(P\).
We construct a family of scaled Frobenius companion linearizations, each designed to control the growth of backward errors when recovering polynomial eigenpairs in a particular region of the spectrum.

When $s\geq 2$, consider each interval $[\alpha_i,\alpha_{i+1}]$ between consecutive tropical roots, for $i=1,\ldots,s-1$. 
We scale the matrix polynomial at the geometric mean of two consecutive tropical roots
\[
\gamma=\sqrt{\alpha_i\alpha_{i+1}}
\]
and construct two scaled companion pencils: one for eigenvalues with moduli in $[\alpha_i,\gamma]$ and the other for those in $[\gamma,\alpha_{i+1}]$.
The first lower-half pencil also covers moduli below $\alpha_1$, while the last upper-half pencil covers moduli above $\alpha_s$.
Thus, the resulting $2(s-1)$ pencils cover all nonzero eigenvalue moduli.
When $s=1$, a single scaled companion pencil suffices.
Since $s\leq d$, the construction requires at most $2(d-1)$ pencils.
  
Our main theoretical result shows that, for both right and left eigenpairs, the coefficientwise relative backward error of a recovered polynomial eigenpair is bounded by the backward error of the corresponding pencil eigenpair times a factor of order \(d^{5/2}\), independent of the coefficient norms. 
This bound holds for exact pencil eigenvectors in their assigned spectral regions.

For computed pencil eigenpairs, the same bound holds provided the norm of the pencil eigenvector remains comparable to that of the block used for recovery. 
Establishing this property for computed eigenvectors remains an open problem, even for quadratic matrix polynomials.

Our analysis does not yet establish a coefficientwise backward stable eigensolver. 
One remaining challenge is to combine the eigenpairs computed from different pencils into a complete set, with every eigenvalue represented according to its multiplicity. The difficulty arises near the boundaries between spectral regions, where different pencils may produce approximations to the same eigenvalue that lie on opposite sides of a boundary. 
As a result, selecting eigenpairs according to their moduli may cause some eigenvalues to be missed and others to be counted more than once. 
Zeng and Su~\cite[Section~3.5]{ZengSu2014} propose a heuristic selection strategy for their two-linearization method, but it does not guarantee that all eigenvalues are recovered.

Our method is also computationally expensive, requiring up to $2(s-1)$ generalized eigenvalue solves of size $dn$ when $s\geq 2$.
A less expensive alternative uses a single tropically scaled companion pencil, but its coefficientwise backward stability analysis relies on a special graded backward error from QZ \cite{VanBarelTisseur2018,TisseurVanBarel2021,CoefficientwiseGradedPerturbations}.
Whether the standard QZ algorithm produces pencil perturbations with this graded structure remains an open question.

Despite these limitations, our approach establishes uniform bounds relating pencil backward errors to coefficientwise relative polynomial backward errors for exact eigenvectors throughout the finite, nonzero spectrum. 
These bounds depend only on the polynomial degree, not on the norms of its coefficients. 
Our numerical experiments illustrate the coefficientwise backward errors obtained with the proposed linearization family and compare them with those obtained using the unscaled companion pencil, tropical scaling, and tropical-scaled Lagrange and companion linearizations.

\section{Preliminaries}

Throughout the paper, $\|\cdot\|_2$ denotes the Euclidean norm for vectors and the spectral norm for matrices. 
We write $I_n$ for the $n\times n$ identity matrix and $e_j$ for the $j$th column of the identity matrix, with the dimension understood from the context. 
The superscript $*$ denotes the conjugate transpose.

In this section, we review the concepts and results needed in the subsequent sections. 

\subsection{Strong linearizations}
\label{sec:strong-linearizations}

Let $P(\lambda)$ be a regular matrix polynomial of degree $d\geq1$, as in \eqref{eq:matrix-polynomial}.
Linearization replaces its polynomial eigenvalue problem by a larger generalized eigenvalue problem.

\begin{definition}
\label{def:strong-linearization}
A $dn\times dn$ pencil $L(\lambda)=\lambda X+Y$ is a \emph{linearization} of $P$ if
\[
    U(\lambda)L(\lambda)V(\lambda)
    =\operatorname{diag}\bigl(P(\lambda),I_{(d-1)n}\bigr)
\]
for some unimodular matrix polynomials $U$ and $V$, that is, matrix polynomials whose determinants are nonzero constants.

A linearization $L(\lambda)$ is \emph{strong} if $\operatorname{rev}_1L$ also linearizes $\operatorname{rev}_dP$, where
\[
    \operatorname{rev}_1L(\lambda)=X+\lambda Y,
    \qquad
    \operatorname{rev}_dP(\lambda)
    =\sum_{k=0}^{d}P_{d-k}\lambda^k.
\]
\end{definition}

A linearization preserves the finite elementary divisors of $P$, and a strong linearization also preserves its infinite elementary divisors. 
Thus, it preserves the eigenvalues and their partial multiplicities \cite{DopicoEtAl2018,MackeyEtAl2006}.

The best-known strong linearization of a matrix polynomial \(P(\lambda)\) as in \eqref{eq:matrix-polynomial} is the \emph{first Frobenius companion pencil}, or simply the \emph{companion pencil},
\begin{equation}
\label{eq:first-companion-pencil}
    C_P(\lambda)
    =
    \lambda
    \begin{bmatrix}
        P_d&&&\\
        &I_n&&\\
        &&\ddots&\\
        &&&I_n
    \end{bmatrix}
    +
    \begin{bmatrix}
        P_{d-1}&P_{d-2}&\cdots&P_0\\
        -I_n&0&\cdots&0\\
        &\ddots&\ddots&\vdots\\
        &&-I_n&0
    \end{bmatrix},
\end{equation}
where blank blocks are understood to be zero.

For every finite \(\lambda\), a nonzero vector \(z\) satisfies \(C_P(\lambda)z=0\) if and only if
\[
    z=
    \begin{bmatrix}
        \lambda^{d-1}x\\
        \lambda^{d-2}x\\
        \vdots\\
        x
    \end{bmatrix}
\]
for some nonzero vector \(x\) satisfying \(P(\lambda)x=0\).
When \(\lambda\neq0\), every block of \(z\) is a nonzero scalar multiple of \(x\), and hence is an eigenvector of \(P\) corresponding to \(\lambda\).
Thus, any block can be used to recover a polynomial eigenvector.

The choice of recovery block is important for backward stability, since it can significantly affect the backward error of the recovered eigenpair.
Traditionally, recovery has been performed using the first or last block, which have the largest norms when \(|\lambda|\geq1\) or \(|\lambda|\leq1\), respectively.
However, as we will show, recovering eigenvectors from intermediate blocks of appropriately scaled companion pencils can lead to smaller backward errors. This additional flexibility is an important ingredient of our approach.

%

\subsection{One-sided factorizations and eigenvector recovery}
\label{subsec:factorizations-recovery}

Higham, Li, and Tisseur~\cite{HighamLiTisseur2007} use \emph{one-sided factorizations} to relate the eigenvectors and residuals of a matrix polynomial to those of its linearizations.
We recall these factorizations in the form needed for our analysis.

Let $L(\lambda)$ be a $dn\times dn$ matrix pencil. 
Suppose that, for \(\lambda\) in a region \(\Omega\subseteq\mathbb C\),
\begin{equation}
\label{eq:one-sided-factorizations}
    G(\lambda)L(\lambda)
       =P(\lambda)E_{\rm R},
    \qquad
    L(\lambda)F(\lambda)
       =E_{\rm L}P(\lambda),
\end{equation}
where $G(\lambda)\in\mathbb C^{n\times dn}$, $F(\lambda)\in\mathbb C^{dn\times n}$, and \(E_{\rm R}\in\mathbb C^{n\times dn}\) and \(E_{\rm L}\in\mathbb C^{dn\times n}\).

The factorizations \eqref{eq:one-sided-factorizations} allow us to recover eigenvectors of \(P\) from those of \(L\) and relate the corresponding residuals for arbitrary vectors.

\begin{proposition}
\label{prop:one-sided-eigenvector-recovery}
Let \(\lambda\in\Omega\), and suppose that
\eqref{eq:one-sided-factorizations} holds. For arbitrary vectors
\(z,w\), set
\[
    x=E_{\rm R}z,
    \qquad
    y=E_{\rm L}^*w.
\]
If \(L(\lambda)z=0\), then \(P(\lambda)x=0\). 
If \(w^*L(\lambda)=0\), then \(y^*P(\lambda)=0\). 
For arbitrary \(z,w\), the residuals satisfy
\[
    \|P(\lambda)x\|_2
       \leq
       \|G(\lambda)\|_2\,\|L(\lambda)z\|_2,
    \qquad
    \|y^*P(\lambda)\|_2
       \leq
       \|F(\lambda)\|_2
       \|w^*L(\lambda)\|_2.
\]
\end{proposition}

Proposition~\ref{prop:one-sided-eigenvector-recovery} shows that if \(z\) and \(w\) are right and left eigenvectors of \(L\), respectively, then \(x=E_{\rm R}z\) and \(y=E_{\rm L}^*w\) are eigenvectors of \(P\), provided that \(x,y\neq 0\).

We recall the well-known one-sided factorizations for the first Frobenius companion pencil \cite{HighamLiTisseur2007}.

\begin{proposition}
\label{prop:companion-one-sided-factorizations}
Let \(P(\lambda)=\sum_{k=0}^{d}P_k\lambda^k\), and let \(C_P(\lambda)\) be its companion pencil.
For $j=1,\hdots,d$, define
\(G^{P}_j(\lambda)
=[G^{P}_{j,1}(\lambda),\ldots,G^{P}_{j,d}(\lambda)]\), where
\[
G^P_{j,i}(\lambda)=
\begin{cases}
\lambda^{d-j}I_n,
    & i=1,\\[1mm]
\lambda^{d-j}
\displaystyle\sum_{k=d-i+1}^{d}
    \lambda^{k-d+i-1}P_k,
    & 2\leq i\leq j,\\[4mm]
-\lambda^{i-j-1}
\displaystyle\sum_{k=0}^{d-i}\lambda^kP_k,
    & j+1\leq i\leq d.
\end{cases}
\]
Then, for every \(\lambda\in\mathbb C\),
\begin{equation}
\label{eq:companion-right-factorization}
    G^P_j(\lambda)C_P(\lambda)
       =P(\lambda)(e_j^T \otimes I_n),
       \qquad j=1,\ldots,d.
\end{equation}

Define
\[
    F(\lambda)=
    \begin{bmatrix}
        \lambda^{d-1}I_n\\
        \lambda^{d-2}I_n\\
        \vdots\\
        \lambda I_n\\
        I_n
    \end{bmatrix}.
\]
Then, for every \(\lambda\in\mathbb C\),
\begin{equation}
\label{eq:companion-left-factorization}
    C_P(\lambda)F(\lambda)=(e_1\otimes I_n)P(\lambda).
\end{equation}

\end{proposition}

If we partition the right and left eigenvectors of the companion pencil \(C_P(\lambda)\) into \(d\) blocks of size \(n\), then Proposition~\ref{prop:companion-one-sided-factorizations} shows that we can recover a right eigenvector of \(P\) from any block of a right eigenvector of \(C_P(\lambda)\) when $\lambda\neq 0$, and a left eigenvector of \(P\) from the first block of a left eigenvector of \(C_P(\lambda)\).

The one-sided factorizations in Proposition~\ref{prop:companion-one-sided-factorizations} also provide explicit expressions for the right and left eigenvectors of the companion pencil.

\begin{proposition}
\label{prop:companion-eigenvectors}
Let $C_P(\lambda)$ be the  companion pencil of $P(\lambda)$, and let $F(\lambda)$ and $G_j^P(\lambda)$ be the factors defined in Proposition~\ref{prop:companion-one-sided-factorizations}. 
For every $\lambda\in\mathbb C$, a nonzero vector $z$ satisfies $C_P(\lambda)z=0$ if and only if
\[
z=F(\lambda)x
\]
for some nonzero $x$ satisfying $P(\lambda)x=0$.
Similarly, for \(j=1,\ldots,d\) and \(\lambda\neq 0\), a nonzero vector \(w\) satisfies \(w^*C_P(\lambda)=0\) if and only if
\[
w^*=y^*G_j^P(\lambda)
\]
for some nonzero $y$ satisfying $y^*P(\lambda)=0$. 
At $\lambda=0$, a nonzero vector $w$ satisfies $w^*C_P(0)=0$ if and only if $w^*=y^*G_d^P(0)$
for some nonzero $y$ satisfying $y^*P(0)=0$.
\end{proposition}

\subsection{Backward errors of computed eigenpairs}
\label{sec:backward-errors}

For a regular matrix polynomial $P$ as in \eqref{eq:matrix-polynomial}, a computed eigenpair $(\lambda,x)$ generally has a nonzero residual $r=P(\lambda)x$.
Its backward error measures the smallest perturbation to the polynomial's coefficients that makes the computed eigenpair exact.
We use the  definitions of Tisseur \cite{Tisseur2000}; see also \cite{HighamLiTisseur2007}.
We use the same notation for exact and computed eigenpairs, distinguishing them in the text when needed.

\begin{definition}
\label{def:weighted-backward-error}
Let $w_0,\ldots,w_d$ be nonnegative weights.
The \emph{backward error} of a computed eigenpair $(\lambda,x)$, with $x\neq0$, is
\begin{equation}
\label{eq:weighted-backward-error}
    \eta_{P,w}(\lambda,x)
    =
    \inf\left\{
    \varepsilon\geq0:
    \begin{array}{l}
    (P(\lambda)+\Delta P(\lambda))x=0,\\[1mm]
    \displaystyle
    \Delta P(\lambda)=\sum_{k=0}^{d}\Delta P_k\lambda^k,
    \qquad \Delta P_k\in\mathbb C^{n\times n},\\[1mm]
    \|\Delta P_k\|_2\leq\varepsilon w_k,
        \qquad k=0,\ldots,d
    \end{array}
    \right\}.
\end{equation}
%
\end{definition}

The weights determine how we measure the coefficient perturbations.
If \(w_k=0\), no perturbation of \(P_k\) is allowed; that is, \(\Delta P_k=0\).

%

The backward error has an explicit formula in terms of the residual \cite{Tisseur2000}.
\begin{proposition}
\label{prop:backward-error-residual}
For $x\neq0$, the backward error satisfies
\begin{equation}
\label{eq:weighted-residual-formula}
    \eta_{P,w}(\lambda,x)
    =
    \frac{\|P(\lambda)x\|_2}
         {\|x\|_2 \sum_{k=0}^{d}w_k|\lambda|^k }.
\end{equation}
\end{proposition}

%

Throughout the paper, we choose the weights \(w_k=\|P_k\|_2\), so that the perturbation bounds become
\[
\|\Delta P_k\|_2\leq\varepsilon\|P_k\|_2,
\qquad 0\leq k\leq d.
\]
Thus, a small backward error means that the computed eigenpair is an exact eigenpair of a nearby matrix polynomial whose coefficients differ from those of \(P\) by small relative perturbations. 
The corresponding \emph{coefficientwise relative backward error} is
\begin{equation}
\label{eq:polynomial-backward-error}
\eta_P(\lambda,x)
=
\frac{\|P(\lambda)x\|_2}
{\|x\|_2
\displaystyle\sum_{k=0}^{d}\|P_k\|_2|\lambda|^k}.
\end{equation}

For a computed left eigenpair $(\lambda,y)$, with $y\neq0$, the analogous coefficientwise relative backward error is
\begin{equation}
\label{eq:left-polynomial-backward-error}
    \eta_P^{\mathrm{left}}(\lambda,y)
    =
    \frac{\|y^*P(\lambda)\|_2}
    {\|y\|_2
     \displaystyle\sum_{k=0}^{d}\|P_k\|_2|\lambda|^k}.
\end{equation}


\subsection{Scaling and reversal}
\label{subsec:scaling-reversal}

Scaling and reversal are the two main operations used to construct our linearizations.
We introduce these operations and show that they preserve the linearization property and the  coefficientwise relative backward errors of corresponding eigenpairs.

We scale a matrix polynomial \(P\) by rescaling its variable and multiplying it by a nonzero scalar.
For \(\gamma>0\) and \(\delta\neq 0\), we define
\[
    Q(\lambda)=\delta P(\gamma\lambda)
    =\sum_{k=0}^{d}\delta\gamma^kP_k\lambda^k.
\]

Reversal reverses the order of the coefficients of \(P\). 
Its reversal with respect to grade \(d\) is defined by
\[
    R(\lambda)=\operatorname{rev}_d P(\lambda)
    =\lambda^dP(\lambda^{-1})
    =\sum_{k=0}^{d}P_{d-k}\lambda^k.
\]

Scaling and reversal give simple correspondences between eigenpairs.
A right or left eigenpair \((\lambda,v)\) of \(Q\) corresponds to the eigenpair \((\gamma\lambda,v)\) of \(P\).
Similarly, an eigenpair \((\lambda,v)\) of \(R\), with \(\lambda\neq0\), corresponds to \((\lambda^{-1},v)\) of \(P\).
In both cases, the eigenvectors remain unchanged.

Scaling and reversal also preserve the coefficientwise relative backward errors of corresponding approximate eigenpairs.

\begin{lemma}
\label{lem:scaling-reversal-backward-error}
Let \(P(\lambda)=\sum_{k=0}^{d}P_k\lambda^k\) be a regular matrix polynomial. 
For \(\gamma>0\) and \(\delta \neq 0\), define
\[
    Q(\lambda)=\delta P(\gamma\lambda),
    \qquad
    R(\lambda)=\operatorname{rev}_d P(\lambda)
        =\sum_{k=0}^{d}P_{d-k}\lambda^k.
\]
For nonzero vectors \(x,y\in\mathbb C^n\),
\begin{align*}
    \eta_Q(\lambda,x)
        &=\eta_P(\gamma\lambda,x),
    &
    \eta_Q^{\mathrm{left}}(\lambda,y)
        &=\eta_P^{\mathrm{left}}(\gamma\lambda,y),
    &&\lambda\in\mathbb C,
    \\
    \eta_R(\lambda,x)
        &=\eta_P(\lambda^{-1},x),
    &
    \eta_R^{\mathrm{left}}(\lambda,y)
        &=\eta_P^{\mathrm{left}}(\lambda^{-1},y),
    &&\lambda\neq0.
\end{align*}
\end{lemma}

%
%

Thus, bounds proved for a scaled or reversed polynomial apply to the original polynomial at the corresponding eigenvalue. 

Scaling and reversal also preserve the strong linearization property. 

\begin{lemma}
\label{lem:linearization-scaling-reversal}
Let \(L(\lambda)\) be a linearization of a matrix polynomial \(P(\lambda)\) as in \eqref{eq:matrix-polynomial}.
 For \(\gamma>0\) and nonzero scalar $\delta\neq 0$, define
\[
    Q(\lambda)=\delta P(\gamma\lambda),
    \qquad
    M(\lambda)= L(\gamma\lambda).
\]
Then \(M\) is a linearization of \(Q\). 
If \(L\) is strong, then \(M\) is strong, and \(\operatorname{rev}_1 L\) is a strong linearization of \(\operatorname{rev}_d P\).
\end{lemma}

Lemmas \ref{lem:scaling-reversal-backward-error} and \ref{lem:linearization-scaling-reversal} let us work with scaled and reversed polynomials and then express the results for the original polynomial. 
Strong linearizations remain strong, and the backward-error bounds retain the same constants.

For example, let $Q(\lambda)=\delta P(\gamma\lambda)$ and  $M(\lambda)= L(\gamma\lambda)$.
If, for nonzero vectors \(x\) and \(z\), we prove
\[
    \eta_Q(\lambda,x)
    \leq C\,\eta_M(\lambda,z),
\]
then Lemma \ref{lem:scaling-reversal-backward-error} gives
\[
    \eta_P(\gamma\lambda,x)
    \leq C\,\eta_L(\gamma\lambda,z).
\]
Thus, scaling preserves the bound and its constant.
The same argument applies to left vectors. 
For reversal, the corresponding bound for the reversed polynomial and pencil at \(\lambda\) gives a bound for the original polynomial and pencil at \(\lambda^{-1}\).
%
%

\subsection{Backward error growth under linearization}\label{sec:backward-error transfer}

Suppose we solve a polynomial eigenvalue problem using a linearization \(L\) of \(P\) that satisfies the one-sided factorizations \eqref{eq:one-sided-factorizations}. 
Given a computed pencil eigenpair \((\lambda,z)\), we recover an approximate eigenvector \(x=E_{\rm R}z\) of \(P\). 
We use the one-sided factorizations to relate the backward errors of \((\lambda,x)\) and \((\lambda,z)\) \cite{HighamLiTisseur2007}.

\begin{proposition}
\label{prop:backward-error-transfer}
Let \(P(\lambda)=\sum_{k=0}^{d}P_k\lambda^k\), and let \(L(\lambda)=\lambda X + Y\) be a pencil. 
Define
\[
    W_P(r)=\sum_{k=0}^{d}\|P_k\|_2r^k,
    \qquad
    W_L(r)=\|Y\|_2+r\|X\|_2,
    \qquad r\geq0.
\]
Suppose the factorizations \eqref{eq:one-sided-factorizations} hold at \(\lambda\), with
\(W_P(|\lambda|)>0\) and \(W_L(|\lambda|)>0\).

For \(z\neq0\) and \(x=E_{\rm R}z\neq0\),
\begin{equation}
\label{eq:right-backward-error-transfer}
    \eta_P(\lambda,x)
    \leq
    \frac{W_L(|\lambda|)\|G(\lambda)\|_2}
         {W_P(|\lambda|)}
    \frac{\|z\|_2}{\|x\|_2}
    \eta_L(\lambda,z).
\end{equation}
Similarly, for \(w\neq0\) and \(y=E_{\rm L}^*w\neq0\),
\begin{equation}
\label{eq:left-backward-error-transfer}
    \eta_P^{\mathrm{left}}(\lambda,y)
    \leq
    \frac{W_L(|\lambda|)\|F(\lambda)\|_2}
         {W_P(|\lambda|)}
    \frac{\|w\|_2}{\|y\|_2}
    \eta_L^{\mathrm{left}}(\lambda,w).
\end{equation}
\end{proposition}

Each of the bounds \eqref{eq:right-backward-error-transfer} and \eqref{eq:left-backward-error-transfer} contains two factors, which we analyze separately.

The first factor depends on the matrix polynomial, the pencil, and the chosen one-sided factorization. We write
\[
g_{\rm R}(\lambda)
=
\frac{W_L(|\lambda|)\|G(\lambda)\|_2}
     {W_P(|\lambda|)},
\qquad
g_{\rm L}(\lambda)
=
\frac{W_L(|\lambda|)\|F(\lambda)\|_2}
     {W_P(|\lambda|)}.
\]
The second factor measures the ratio between the norm of the pencil eigenvector and that of the recovered polynomial eigenvector.
 We call these the \emph{recovery norm ratios}:
\[
\rho_{\rm R}(z)
=\frac{\|z\|_2}{\|E_{\rm R}z\|_2},
\qquad
\rho_{\rm L}(w)
=\frac{\|w\|_2}{\|E_{\rm L}^*w\|_2},
\]
provided the denominators are nonzero.

%
%

\section{A family of scaled companion linearizations}

We now construct the family of scaled companion linearizations used in our approach.
The tropical roots determine a partition of the positive real axis and provide the scaling parameters for each region. 
We first use these parameters to construct scaled matrix polynomials whose coefficient norms satisfy suitable bounds. 
We then apply diagonal row and column scalings to their Frobenius companion pencils. 
Finally, we express the resulting linearizations in the original eigenvalue variable and assign each one a spectral region and eigenvector recovery rules. 
The next section shows that these choices lead to backward-error amplification bounds that depend only on the polynomial degree.

\subsection{Tropical roots}
\label{sec:tropical-roots}
For a matrix polynomial $P$ as in \eqref{eq:matrix-polynomial}, let \(p_k=\|P_k\|_2\), and assume \(p_0,p_d>0\). The \emph{tropical polynomial} of \(P\) is
\[
T_P(r)=\max_{0\leq k\leq d}p_kr^k,
\qquad r>0.
\]
A \emph{tropical root} is a radius \(\alpha>0\) at which the term attaining the maximum in \(T_P(r)\) changes as \(r\) passes through \(\alpha\).

The tropical roots can be computed from the \emph{upper convex hull} of the points
\[
(k,\log p_k),\qquad p_k>0.
\]
Let
\[
0=v_0<v_1<\cdots<v_s=d
\]
be the indices of the vertices of this hull, omitting points that lie in the interior of a hull segment.
Then \(P\) has \(s\) distinct tropical roots
\begin{equation}
\label{eq:tropical-hull-roots}
\alpha_i
=
\left(\frac{p_{v_{i-1}}}{p_{v_i}}\right)^{
1/(v_i-v_{i-1})},
\qquad i=1,\ldots,s,
\end{equation}
with
\[
0<\alpha_1<\cdots<\alpha_s.
\]
For $i=1,\ldots,s,$ the multiplicity of \(\alpha_i\) is
\[
m_i=v_i-v_{i-1},
 \qquad \sum_{i=1}^s m_i=d.
\]

\begin{example}\label{ex:tropical roots example}
Consider a degree-five matrix polynomial with coefficient
norms
\[
    (p_0,p_1,p_2,p_3,p_4,p_5)
    =(1,10,10^4,10^5,10^6,10^4).
\]
Using base-ten logarithms, the coefficient points are
\[
    (k,\log_{10}p_k)
    =(0,0),(1,1),(2,4),(3,5),(4,6),(5,4).
\]
The upper hull has vertex indices $0,2,4,5$.
Its segments have slopes $2$, $1$, and $-2$, so the
distinct tropical roots are
\[
    \alpha_1=10^{-2},\qquad
    \alpha_2=10^{-1},\qquad
    \alpha_3=10^2,
\]
with multiplicities $2$, $2$, and $1$, respectively.

The point at index $1$ lies below the hull, so its term never attains the maximum. 
The point at index $3$ lies inside the segment joining indices $2$ and $4$.
Its term attains the maximum at $\alpha_2$, but does not create another tropical root.
Figure~\ref{fig:tropical-degree-five} illustrates these two cases.


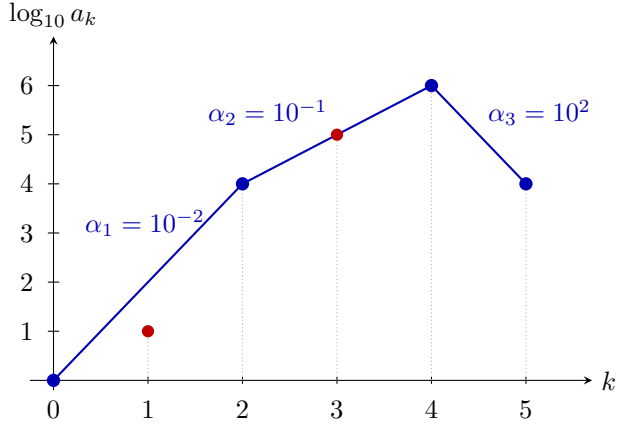
\begin{figure}[htbp]
\centering
\begin{tikzpicture}[x=1.25cm,y=0.65cm,>=stealth]
\draw[->] (-0.25,0) -- (5.7,0) node[right] {$k$};
\draw[->] (0,-0.2) -- (0,7)
    node[above] {$\log_{10}p_k$};

\foreach \k in {0,...,5}{
    \draw (\k,0.07) -- (\k,-0.07);
    \node[below=4pt] at (\k,0) {$\k$};
}
\foreach \b in {1,...,6}{
    \draw (-0.06,\b) -- (0.06,\b);
    \node[left=4pt] at (0,\b) {$\b$};
}

\foreach \k/\b in {1/1,2/4,3/5,4/6,5/4}
    \draw[gray!60,densely dotted] (\k,0) -- (\k,\b);

\draw[blue!70!black,thick]
    (0,0) -- (2,4) -- (4,6) -- (5,4);

\node[blue!70!black,above left] at (1.7,2.75)
    {$\alpha_1=10^{-2}$};
\node[blue!70!black,above left] at (3,5)
    {$\alpha_2=10^{-1}$};
\node[blue!70!black,above right] at (4.5,5)
    {$\alpha_3=10^2$};

\foreach \k/\b in {0/0,2/4,4/6,5/4}
    \fill[blue!70!black] (\k,\b) circle (2.5pt);

\foreach \k/\b in {1/1,3/5}
    \fill[red!75!black] (\k,\b) circle (2.3pt);

\end{tikzpicture}
\caption{Upper convex hull of $(k,\log_{10}p_k)$ for a degree-five matrix polynomial example. 
A segment of slope $\sigma$ gives a tropical root $10^{-\sigma}$; its horizontal length gives the multiplicity.
The red points lie below the hull or inside a hull segment.}
\label{fig:tropical-degree-five}
\end{figure}
\end{example}

For each \(j\) with \(p_j>0\), define the \emph{dominance interval}
\[
I_j=\{r>0:p_jr^j=T_P(r)\}.
\]

\begin{proposition}
\label{prop:tropical-dominance}
Let \(P\) have distinct tropical roots \(0<\alpha_1<\cdots<\alpha_s\), with corresponding upper convex hull vertex indices \(0=v_0<\cdots<v_s=d\). 
Then
\[
I_0=(0,\alpha_1],
\qquad
I_d=[\alpha_s,\infty).
\]
For \(1\leq j\leq d-1\) with \(p_j>0\),
\[
I_j=
\begin{cases}
[\alpha_i,\alpha_{i+1}],
& j=v_i,\quad 1\leq i\leq s-1,\\[1mm]
\{\alpha_i\},
& (j,\log p_j)\text{ lies in the interior of the hull segment}\\
& \text{joining }(v_{i-1},\log p_{v_{i-1}})
\text{ and }(v_i,\log p_{v_i}),\\[1mm]
\varnothing,
& (j,\log p_j)\text{ lies strictly below the upper hull}.
\end{cases}
\]
\end{proposition}

%
%
%

\begin{example}
For the degree-five example in Example \ref{ex:tropical roots example}, the hull vertices have indices \(0,2,4,5\), with dominance intervals
\[
    I_0=(0,10^{-2}],\quad
    I_2=[10^{-2},10^{-1}],\quad
    I_4=[10^{-1},10^2],\quad
    I_5=[10^2,\infty).
\]
The point at index \(1\) lies below the hull, so \(I_1=\varnothing\).
The point at index \(3\) lies inside the hull segment joining indices \(2\) and \(4\), so \(I_3=\{10^{-1}\}\). 
At this radius, the terms \(p_2r^2\), \(p_3r^3\), and \(p_4r^4\) all attain \(T_P(r)\).
\end{example}


The dominance intervals of the upper-hull vertices cover all positive radii:
\[
\bigcup_{i=0}^{s} I_{v_i}=(0,\infty).
\]
Thus, for every \(r>0\), at least one upper-hull vertex term attains the maximum \(T_P(r)\).

\subsection{Tropical scaling of matrix polynomials}
\label{subsec:polynomial-scaling}

We use tropical roots to scale the matrix polynomial, but our approach differs from the usual practice of scaling at the tropical roots themselves \cite{GaubertSharify2009,NoferiniSharifyTisseur2015}. Instead, we scale at the geometric means of consecutive tropical roots.

Assume first that \(T_P\) has \(s\geq 2\) distinct tropical roots $\alpha_1<\cdots<\alpha_s$.
For \(1\leq i\leq s-1\), let \(j=v_i\) (with $1\leq j\leq d-1$) be an interior upper-hull vertex, with corresponding dominance interval \(I_j=[\alpha_i,\alpha_{i+1}]\). 
Define
\begin{equation}
\label{eq:tropical-midpoint-scale}
\gamma_j=\sqrt{\alpha_i\alpha_{i+1}},
\qquad
U_j=\sqrt{\frac{\alpha_{i+1}}{\alpha_i}}.
\end{equation}
Since the \(j\)th term dominates on \(I_j\), we have \(T_P(\gamma_j)=p_j\gamma_j^j\). 
We then scale \(P\) to obtain
\begin{equation}
\label{eq:tropical-midpoint-polynomial}
Q_j(\lambda)
=\frac{P(\gamma_j\lambda)}{T_P(\gamma_j)}
=\sum_{k=0}^{d}Q_{j,k}\lambda^k,
\qquad
Q_{j,k}=\frac{\gamma_j^{k-j}}{p_j}P_k,
\end{equation}
where \(p_j=\|P_j\|_2\).
Since \(P_0\neq0\) and \(P_d\neq0\), the endpoint coefficients of \(Q_j\) are also nonzero. 
Thus, both \(Q_j\) and \(\operatorname{rev}_d Q_j\) have degree \(d\).

In the scaled variable, the dominance interval \(I_j=[\alpha_i,\alpha_{i+1}]\) becomes
\[
U_j^{-1}\leq |\lambda|\leq U_j.
\]

The following proposition shows that the scaling \eqref{eq:tropical-midpoint-polynomial} makes the \(j\)th coefficient norm one and bounds the other coefficient norms by a geometrically decreasing function of their distance from \(j\).

\begin{proposition}
\label{propo:tropical-coefficient-decay}
Let \(j=v_i\) be an interior upper-hull vertex, \(1\leq i\leq s-1\), and let \(Q_j\), \(\gamma_j\), and \(U_j\) be defined by \eqref{eq:tropical-midpoint-polynomial} and \eqref{eq:tropical-midpoint-scale}. Then
\begin{equation}
\label{eq:tropical-coefficient-decay}
\|Q_{j,j}\|_2=1,
\qquad
\|Q_{j,k}\|_2\leq U_j^{-|k-j|},
\qquad 0\leq k\leq d.
\end{equation}
Moreover,
\[
\|Q_{j,0}\|_2=U_j^{-j}
\quad\text{if }j=v_1,
\qquad
\|Q_{j,d}\|_2=U_j^{-(d-j)}
\quad\text{if }j=v_{s-1}.
\]
\end{proposition}

\begin{proof}
Write \(p_k=\|P_k\|_2\).
For \(k<j\), the \(j\)th term attains the maximum at \(\alpha_i\), so
\(p_k\alpha_i^k\leq p_j\alpha_i^j\).
Hence
\[
\|Q_{j,k}\|_2
=\frac{p_k}{p_j}\gamma_j^{k-j}
\leq \left(\frac{\alpha_i}{\gamma_j}\right)^{j-k}
=U_j^{-(j-k)}.
\]
For \(k>j\), dominance at \(\alpha_{i+1}\) gives
\(p_k\alpha_{i+1}^k\leq p_j\alpha_{i+1}^j\), and therefore
\[
\|Q_{j,k}\|_2
\leq \left(\frac{\gamma_j}{\alpha_{i+1}}\right)^{k-j}
=U_j^{-(k-j)}.
\]
Also, \(\|Q_{j,j}\|_2=1\) by definition.

If \(j=v_1\), the definition of \(\alpha_1\) gives
\(p_0=p_j\alpha_1^j\). Thus,
\[
\|Q_{j,0}\|_2
=\frac{p_0}{p_j\gamma_j^j}
=\left(\frac{\alpha_1}{\gamma_j}\right)^j
=U_j^{-j}.
\]
Similarly, if \(j=v_{s-1}\), then
\(p_d\alpha_s^d=p_j\alpha_s^j\), so
\[
\|Q_{j,d}\|_2
=\frac{p_d}{p_j}\gamma_j^{d-j}
=\left(\frac{\gamma_j}{\alpha_s}\right)^{d-j}
=U_j^{-(d-j)}.
\]
\end{proof}

For the scaled polynomial \(Q_j\), the sum appearing in the denominator of the backward error formula satisfies
\begin{equation}
\label{eq:scaled-polynomial-weight}
W_{Q_j}(r)
=\sum_{k=0}^{d}\|Q_{j,k}\|_2r^k
=\frac{W_P(\gamma_j r)}{T_P(\gamma_j)},
\qquad r\geq 0.
\end{equation}

The following proposition bounds \(W_{Q_j}(r)\) by constant multiples of \(r^j\) on the rescaled dominance interval, of \(1\) on the inner region, and of \(r^d\) on the outer region.

\begin{proposition}
\label{proposition:tropical-polynomial-scaling}
Let \(Q(\lambda)=\sum_{k=0}^d Q_k\lambda^k\), and suppose that, for some \(1\leq j\leq d-1\) and \(U\geq 1\),
\[
\|Q_j\|_2=1,
\qquad
\|Q_k\|_2\leq U^{-|k-j|},
\qquad 0\leq k\leq d.
\]
Then
\[
r^j\leq W_Q(r)\leq(d+1)r^j,
\qquad U^{-1}\leq r\leq U.
\]
Moreover, if \(\|Q_0\|_2=U^{-j}\), then
\[
U^{-j}\leq W_Q(r)\leq(d+1)U^{-j},
\qquad 0\leq r\leq U^{-1}.
\]
Similarly, if \(\|Q_d\|_2=U^{-(d-j)}\), then
\[
U^{-(d-j)}r^d\leq W_Q(r)
\leq(d+1)U^{-(d-j)}r^d,
\qquad r\geq U.
\]
\end{proposition}

\begin{proof}
Recall that
\[
T_Q(r)=\max_{0\leq k\leq d}\|Q_k\|_2r^k,
\qquad
T_Q(r)\leq W_Q(r)\leq(d+1)T_Q(r).
\]
We show that the stated assumptions determine \(T_Q(r)\) on each range.

For \(U^{-1}\leq r\leq U\), the coefficient bounds give
\[
\|Q_k\|_2r^k
\leq U^{-|k-j|}r^k
\leq r^j,
\qquad 0\leq k\leq d.
\]
Since \(\|Q_j\|_2=1\), we obtain \(T_Q(r)=r^j\).

Suppose next that \(\|Q_0\|_2=U^{-j}\).
For \(0\leq r\leq U^{-1}\),
\[
\|Q_k\|_2r^k
\leq U^{-|k-j|}U^{-k}
\leq U^{-j},
\qquad 0\leq k\leq d.
\]
Equality holds for \(k=0\), so \(T_Q(r)=U^{-j}\).

Finally, suppose that \(\|Q_d\|_2=U^{-(d-j)}\).
For \(r\geq U\),
\[
\|Q_k\|_2r^k
\leq U^{-|k-j|}r^k
\leq U^{-(d-j)}r^d,
\qquad 0\leq k\leq d.
\]
Equality holds for \(k=d\), so
\(T_Q(r)=U^{-(d-j)}r^d\).
The bounds follow.
\end{proof}

\subsection{Scaled companion pencils}
\label{subsec:scaled-companion-pencils}

We next construct the scaled companion pencils that will be used with the tropical scalings introduced above.
It is convenient to describe this construction for a general matrix polynomial \(Q\) and a scaling parameter \(U\). 
We will later apply it to the polynomials \(Q_j\) and the corresponding parameters \(U_j\).

Let $Q(\lambda)=\sum_{k=0}^d Q_k\lambda^k$ be an \(n\times n\) matrix polynomial of degree \(d\). 
For \(1\leq j\leq d-1\) and \(U>0\), we construct a scaled companion pencil by applying block-diagonal row and column scalings to the Frobenius companion pencil \(C_Q(\lambda)\).

Define
\begin{align}
\label{eq:diagonal-scaling}
D_{R,j}(U)
&=
\operatorname{diag}\left(
\underbrace{U^{d-j-1}I_n,\ldots,UI_n}_{d-j-1\text{ blocks}},
\underbrace{I_n,\ldots,I_n}_{j+1\text{ blocks}}
\right),\\
D_{L,j}(U)
&=
\operatorname{diag}\left(
I_n,
\underbrace{U^{-(d-j-1)}I_n,\ldots,U^{-1}I_n}_{d-j-1\text{ blocks}},
\underbrace{-U^{-1}I_n,\ldots,-U^{-1}I_n}_{j\text{ blocks}}
\right).
\end{align}
When $j=d-1$, the block groups involving powers $U^{d-j-1},\ldots,U$ are empty.

The scaled companion pencil is
\begin{equation}
\label{eq:scaled-companion-definition}
S_j^Q(\lambda)=D_{L,j}(U)C_Q(\lambda)D_{R,j}(U).
\end{equation}
Since $D_{L,j}(U)$ and $D_{R,j}(U)$ are nonsingular, $S_j^Q$ is a strong linearization of $Q$.

For $1\leq j<d-1$, the pencil has the block form
\begin{equation*}
\begingroup
\small
\setlength{\arraycolsep}{2.5pt}
\left[
\begin{array}{cccc|cccc}
U^{d-j-1}(Q_{d-1}+\lambda Q_d)
    & U^{d-j-2}Q_{d-2}
    & \cdots
    & UQ_{j+1}
    & Q_j & Q_{j-1} & \cdots & Q_0\\
\hline
-I_n & \frac{\lambda}{U}I_n & & & & & &\\
& \ddots & \ddots & & & & &\\
& & -I_n & \frac{\lambda}{U}I_n & & & &\\
& & & -I_n & \frac{\lambda}{U}I_n & & &\\
\hline
& & & & \frac{1}{U}I_n & -\frac{\lambda}{U}I_n & &\\
& & & & & \ddots & \ddots &\\
& & & & & & \frac{1}{U}I_n & -\frac{\lambda}{U}I_n
\end{array}
\right].
\endgroup
\end{equation*}
Blank entries are zero.
The vertical divider separates the $d-j-1$ columns associated with $Q_{d-1},\ldots,Q_{j+1}$ from the $j+1$ columns associated with $Q_j,\ldots,Q_0$.
When $j=1$, the block below the second horizontal line consists of a single block row.

If $j=d-1$, then
\[
S_{d-1}^Q(\lambda)=
\begin{bmatrix}
Q_{d-1}+\lambda Q_d &Q_{d-2}&\cdots&Q_1&Q_0\\
\frac{1}{U}I_n&-\frac{\lambda}{U}I_n\\
&\frac{1}{U}I_n&-\frac{\lambda}{U}I_n\\
&&\ddots&\ddots&\ddots\\
&&&\frac{1}{U}I_n&-\frac{\lambda}{U}I_n
\end{bmatrix}.
\]

The following proposition bounds the norms of the pencil coefficients when the coefficients of $Q$ satisfy the bounds established in Proposition~\ref{propo:tropical-coefficient-decay}.

\begin{proposition}
\label{propo:scaled-companion-coefficient-bounds}
Let $Q(\lambda)=\sum_{k=0}^d Q_k\lambda^k$, and suppose that, for some $1\leq j\leq d-1$ and $U\geq1$,
\[
\|Q_j\|_2=1,
\qquad
\|Q_k\|_2\leq U^{-|k-j|},
\qquad 0\leq k\leq d.
\]
Write $S_j^Q(\lambda)=S_{j,0}^Q+\lambda S_{j,1}^Q$.
Then
\[
\|S_{j,0}^Q\|_2\leq\sqrt{d+1},
\qquad
\|S_{j,1}^Q\|_2\leq U^{-1}.
\]
\end{proposition}

\begin{proof}
The coefficient bounds imply that every block in the first block row of $S_{j,0}^Q$ has norm at most $1$.
Since this row contains $d$ blocks, its norm is at most
$\sqrt d$.

The remaining block rows of $S_{j,0}^Q$ contain one nonzero block each, equal to $-I_n$ or $U^{-1}I_n$, in distinct block columns.
Thus, the submatrix formed by these rows has norm at most $1$.
Consequently,
\[
\|S_{j,0}^Q\|_2\leq\sqrt{d+1}.
\]

Similarly, every nonzero block of $S_{j,1}^Q$ has norm at most $U^{-1}$.
These blocks occupy distinct block rows and columns, so
\[
\|S_{j,1}^Q\|_2\leq U^{-1}.
\]
\end{proof}

By Proposition~\ref{propo:tropical-coefficient-decay}, the scaled polynomials $Q_j$ constructed in the preceding subsection satisfy the hypotheses of Proposition~\ref{propo:scaled-companion-coefficient-bounds} with $U=U_j$.

\subsection{The linearization family}
\label{sec:resulting-linearization-family}

We now combine the tropical scaling of the matrix polynomial with the diagonal scaling of its companion pencil to construct our family of linearizations. 
Each pencil is associated with a region of the complex plane and a choice of eigenvector recovery blocks. 
These choices will be used in the next section to establish backward-error amplification bounds that depend only on the polynomial degree.

Assume first that $s\geq2$, and let $\alpha_1<\cdots<\alpha_s$ be the distinct tropical roots.
For each interior upper-hull vertex $j=v_i$, $1\leq i\leq s-1$, recall the scaled polynomial $Q_j$ and the parameters $\gamma_j$ and $U_j$ defined in Section~\ref{subsec:polynomial-scaling}.

We associate two pencils with each such vertex:
\begin{equation}
\label{eq:linearization-family}
\begin{aligned}
L_j^+(\lambda)
&=S_j^{Q_j}(\lambda/\gamma_j)\\
&=D_{L,j}(U_j)\,
  C_{Q_j}(\lambda/\gamma_j)\,
  D_{R,j}(U_j),\\[2mm]
L_j^-(\lambda)
&=\operatorname{rev}_1
  S_{d-j}^{\widetilde Q_j}(\lambda/\gamma_j)\\
&=\left[
  \operatorname{rev}_1
  \left(
  D_{L,d-j}(U_j)\,
  C_{\widetilde Q_j}\,
  D_{R,d-j}(U_j)
  \right)
  \right](\lambda/\gamma_j),
\end{aligned}
\end{equation}
where $\widetilde Q_j=\operatorname{rev}_d Q_j$.
Both pencils are strong linearizations of $P$.

The two pencils cover the dominance interval
$[\alpha_i,\alpha_{i+1}]$.
We assign $L_j^-$ to the lower half,
$\alpha_i\leq|\lambda|\leq\gamma_j$, and
$L_j^+$ to the upper half,
$\gamma_j\leq|\lambda|\leq\alpha_{i+1}$.
The same pencils also cover the two endpoint regions:
$L_{v_1}^-$ is used for $0<|\lambda|\leq\alpha_1$,
and $L_{v_{s-1}}^+$ for $|\lambda|\geq\alpha_s$.

When $s=1$, let $\alpha$ be the unique tropical root and define
\[
Q(t)=\frac{P(\alpha t)}{p_0},
\qquad
L_\alpha(\lambda)=C_Q(\lambda/\alpha).
\]
In this case, a single companion pencil covers the entire
complex plane, including zero.

For each pencil, we also specify the block used to recover
right and left eigenvectors. Table~\ref{tab:exact-recovery-rules}
summarizes the complete construction, including the assigned
regions and recovery blocks.

\begin{table}[htbp]
\centering
\begin{tabular}{llccc}
\toprule
Case & Pencil & Eigenvalue region
     & Right block & Left block\\
\midrule
Interior, upper half
     & $L_j^+$
     & $\gamma_j\leq|\lambda|\leq\alpha_{i+1}$
     & $d-j$ & $1$\\
Interior, lower half
     & $L_j^-$
     & $\alpha_i\leq|\lambda|\leq\gamma_j$
     & $j$ & $1$\\
First endpoint
     & $L_{v_1}^-$
     & $0<|\lambda|\leq\alpha_1$
     & $1$ & $1$\\
Last endpoint
     & $L_{v_{s-1}}^+$
     & $|\lambda|\geq\alpha_s$
     & $1$ & $1$\\
\midrule
Single root
     & $L_\alpha$
     & $|\lambda|\leq\alpha$
     & $d$ & $1$\\
Single root
     & $L_\alpha$
     & $|\lambda|\geq\alpha$
     & $1$ & $1$\\
\bottomrule
\end{tabular}
\caption{Linearization family and eigenvector recovery blocks.
For $s\geq2$, the first four rows cover the nonzero complex plane, with $j=v_i$, $1\leq i\leq s-1$.
For $s=1$, the last two rows cover the entire complex plane.
At shared boundaries, either recovery rule may be used.}
\label{tab:exact-recovery-rules}
\end{table}

Thus, when $s\geq2$, the family consists of $2(s-1)$ pencils covering all nonzero eigenvalue moduli.
When $s=1$, a single pencil covers all finite eigenvalues.

\begin{theorem}
\label{thm:family-amplification}
Let
\[
A_d=\sqrt{d+1}+1,\qquad
B_d=\left(1+\sum_{k=1}^{d-1}k^2\right)^{1/2},
\qquad
\widehat B_d=\left(\sum_{k=1}^{d}k^2\right)^{1/2}.
\]
For each $\lambda$, choose a pencil $L(\lambda)$ and the corresponding right and left recovery blocks according to Table~\ref{tab:exact-recovery-rules}.

Let $z,w\in\mathbb C^{dn}$ be nonzero vectors, and let $x,y\in\mathbb C^n$ be the vectors obtained from the prescribed right and left recovery blocks, respectively.
Assume that $x$ and $y$ are nonzero.

If $s\geq2$, then for every $\lambda\neq0$,
\[
\eta_P(\lambda,x)
\leq
A_dB_d\,
\frac{\|z\|_2}{\|x\|_2}\,
\eta_L(\lambda,z),\qquad
\eta_P^{\mathrm{left}}(\lambda,y)
\leq
A_d\sqrt d\,
\frac{\|w\|_2}{\|y\|_2}\,
\eta_L^{\mathrm{left}}(\lambda,w).
\]

If $s=1$, the bounds hold for every $\lambda\in\mathbb C$, with $B_d$ replaced by $\widehat B_d$. 
Thus,
\[
\eta_P(\lambda,x)
\leq
A_d\widehat B_d\,
\frac{\|z\|_2}{\|x\|_2}\,
\eta_L(\lambda,z),\qquad 
\eta_P^{\mathrm{left}}(\lambda,y)
\leq
A_d\sqrt d\,
\frac{\|w\|_2}{\|y\|_2}\,
\eta_L^{\mathrm{left}}(\lambda,w).
\]
\end{theorem}

The preceding theorem bounds the polynomial backward error in terms of the pencil backward error and the ratio between the norms of the pencil vector and its recovered vector. 
We now show that these ratios are bounded independently of the coefficient norms when the pencil vectors are exact eigenvectors.

\begin{theorem}
\label{thm:family-eigenvector-ratios}
Choose a pencil $L$ and the corresponding right and left recovery blocks according to Table~\ref{tab:exact-recovery-rules}.

Let $\lambda$ be a finite eigenvalue of $P$ in the assigned region, and let $z,w\in\mathbb C^{dn}$ be corresponding nonzero exact right and left eigenvectors of $L$, respectively.
Let $x,y\in\mathbb C^n$ be the vectors recovered from the prescribed right and left blocks.

If $s\geq2$, then $\lambda\neq0$, the recovered vectors $x$ and $y$ are nonzero, and
\[
\frac{\|z\|_2}{\|x\|_2}\leq\sqrt d,
\qquad
\frac{\|w\|_2}{\|y\|_2}\leq B_d.
\]

If $s=1$, the same conclusions hold for every finite eigenvalue $\lambda$, including $\lambda=0$, with
$B_d$ replaced by $\widehat B_d$.
\end{theorem}

The preceding theorem establishes the recovery norm bounds for exact pencil eigenvectors. 
In our numerical experiments, the computed eigenvectors also satisfy these bounds. 
If this holds, combining the two theorems gives
\[
\eta_P(\lambda,x)
\leq A_dB_d\sqrt d\,\eta_L(\lambda,z),
\qquad
\eta_P^{\mathrm{left}}(\lambda,y)
\leq A_dB_d\sqrt d\,\eta_L^{\mathrm{left}}(\lambda,w),
\]
when \(s\geq2\). 
Moreover,
\[
A_dB_d\sqrt d<1.3\,d^{5/2}.
\]
For \(s=1\), the same conclusions hold with \(B_d\) replaced by \(\widehat B_d\), and the amplification factor is bounded by \(2d^{5/2}\). 
Thus, provided the computed eigenvectors satisfy the same recovery norm bounds, the polynomial backward errors are at most a factor \(O(d^{5/2})\) larger than the corresponding pencil backward errors.

\section{Backward-error amplification bounds}

We now prove the backward-error bounds stated in Theorems~\ref{thm:family-amplification} and~\ref{thm:family-eigenvector-ratios}.
The key is to use the coefficient bounds obtained from tropical scaling, together with the structure of the scaled companion pencils, to control how pencil residuals are transferred to polynomial residuals.

We begin with the two halves of the interior dominance intervals. 
We then extend the analysis to the two endpoint regions and finally consider the case of a single tropical root.

\subsection{Interior annuli}

Let \(Q(\lambda)=\sum_{k=0}^{d}Q_k\lambda^k\) be a regular matrix polynomial of degree \(d\geq2\), with \(Q_0\neq0\).
Assume that, for some $1\leq j\leq d-1$ and $U\geq1$, its coefficients satisfy
\begin{equation}
\label{eq:interior-coefficient-assumption}
\|Q_j\|_2=1,
\qquad
\|Q_k\|_2\leq U^{-|k-j|},
\qquad 0\leq k\leq d.
\end{equation}
These are precisely the bounds established in Section~\ref{subsec:polynomial-scaling} for the scaled polynomials $Q_j$, with $U=U_j$.

We first consider the scaled companion pencil
\[
S_j^Q(\lambda)
=
D_{L,j}(U)\,C_Q(\lambda)\,D_{R,j}(U).
\]
Recall the companion one-sided factorizations \eqref{eq:companion-right-factorization} and~\eqref{eq:companion-left-factorization}.
Since $D_{R,j}(U)$ has an $I_n$ block in position $d-j$, and $D_{L,j}(U)$ has an $I_n$ block in its first position,
\[
(e_{d-j}^T\otimes I_n)D_{R,j}(U)
=e_{d-j}^T\otimes I_n,
\qquad
D_{L,j}(U)(e_1\otimes I_n)
=e_1\otimes I_n.
\]
It follows that
\begin{equation}
\label{eq:scaled-companion-factorizations}
\begin{aligned}
\left(
G_{d-j}^Q(\lambda)D_{L,j}(U)^{-1}
\right)S_j^Q(\lambda)
&=Q(\lambda)(e_{d-j}^T\otimes I_n),\\
S_j^Q(\lambda)
\left(D_{R,j}(U)^{-1}F(\lambda)\right)
&=(e_1\otimes I_n)Q(\lambda).
\end{aligned}
\end{equation}

These factorizations show that right eigenvectors of $Q$ can be recovered from block $d-j$ of a right eigenvector of $S_j^Q$, and left eigenvectors from block $1$ of a left eigenvector. 
They also allow us to transfer pencil residuals to polynomial residuals.

We now establish two results for \(1\leq|\lambda|\leq U\). 
The first bounds the backward-error amplification factors for arbitrary pencil vectors, apart from the eigenvector norm ratios. The second bounds these ratios for exact pencil eigenvectors.

\begin{theorem}
\label{thm:scaled-companion-amplification}
Let \(Q(\lambda)=\sum_{k=0}^{d}Q_k\lambda^k\) be a regular matrix polynomial of degree \(d\geq2\), with \(Q_0\neq0\).
Let \(1\leq j\leq d-1\) and \(U\geq1\), and suppose that the coefficients of \(Q\) satisfy \eqref{eq:interior-coefficient-assumption}.
Consider the scaled companion pencil \(S_j^Q\) defined in \eqref{eq:scaled-companion-definition}.
Set
\[
A_d=\sqrt{d+1}+1,
\qquad
B_d=\left(1+\frac{(d-1)d(2d-1)}{6}\right)^{1/2}.
\]
For $1\leq|\lambda|\leq U$, the factors
\begin{align*}
&g_{\rm R}^{Q,+}(\lambda)
=
\frac{
W_{S_j^Q}(|\lambda|)
\|G_{d-j}^{Q}(\lambda)D_{L,j}(U)^{-1}\|_2
}{
W_Q(|\lambda|)
},\qquad \text{and}\\
&g_{\rm L}^{Q,+}(\lambda)
=
\frac{
W_{S_j^Q}(|\lambda|)
\|D_{R,j}(U)^{-1}F(\lambda)\|_2
}{
W_Q(|\lambda|)
}
\end{align*}
satisfy
\[
g_{\rm R}^{Q,+}(\lambda)\leq A_dB_d,
\qquad
g_{\rm L}^{Q,+}(\lambda)\leq A_d\sqrt d.
\]
\end{theorem}

\begin{proof}
For simplicity, we omit superscripts and subscripts throughout the proof and write
\[
S=S_j^Q,\qquad
G=G_{d-j}^Q,\qquad
D_R=D_{R,j},\qquad D_L=D_{L,j}.
\]

Write \(S(\lambda)=S_0+\lambda S_1\). 
By Proposition~\ref{propo:scaled-companion-coefficient-bounds},
\[
\|S_0\|_2\leq\sqrt{d+1},
\qquad
\|S_1\|_2\leq U^{-1}.
\]
Hence, for \(1\leq|\lambda|\leq U\),
\[
W_S(|\lambda|)
\leq \sqrt{d+1}+\frac{|\lambda|}{U}
\leq \sqrt{d+1}+1=A_d.
\]
Moreover, since \(\|Q_j\|_2=1\),
\[
W_Q(|\lambda|)\geq|\lambda|^j.
\]

We next bound the norm of the left factor \(\mathcal F(\lambda):=D_R(U)^{-1}F(\lambda)\). 
We have
\[
\mathcal F(\lambda)
=
D_R(U)^{-1}F(\lambda) = 
\begin{bmatrix}
    U^{-(d-j-1)}\lambda^{d-1}I_n \\
    \vdots \\
    U^{-1}\lambda^{j+1}I_n \\
    \lambda^j I_n \\
    \vdots \\
    I_n
\end{bmatrix}
=
\lambda^j
\begin{bmatrix}
    (\lambda/U)^{d-j-1} I_n \\
    \vdots \\
    (\lambda/U) I_n \\
    I_n \\
    \vdots \\
    \lambda^{-j} I_n
\end{bmatrix}.
\]
where the first \(d-j-1\) blocks are absent when \(d-j=1\).
Since $1\leq |\lambda|\leq U$, each block has norm at most \(|\lambda|^j\), so
\begin{equation}\label{eq:bound norm F}
    \|\mathcal F(\lambda)\|_2\leq\sqrt d\,|\lambda|^j.
\end{equation}

For the right factor, write
\[
    \mathcal G(\lambda)
       =[\mathcal G_{1}(\lambda),\ldots,
          \mathcal G_{d}(\lambda)]:=G(\lambda) D_L(U)^{-1}.
\]
Substituting the blocks of \(G\) into \(\mathcal G=G D_L(U)^{-1}\) gives, for \(1\leq i\leq d\),
\[
\mathcal G_{i}(\lambda)=
\lambda^j\times
\begin{cases}
\displaystyle
I_n, &i=1, \\[4mm]
\displaystyle
 U^{d-j+1-i}\sum_{k=d-i+1}^d \lambda^{k-d+i-1}Q_{k},
    &2\leq i\leq d-j,\\[4mm]
\displaystyle
U\sum_{k=0}^{d-i} \lambda^{k+i-d-1}Q_{k},
    &d-j+1\leq i\leq d.
\end{cases}
\]
For $i=2,\hdots,d-j$, we have
\begin{alignat*}{2}
    \|\mathcal{G}_{i}(\lambda)\|_2 
    &\leq  
    |\lambda|^j \sum_{k=d-i+1}^d U^{d-j+1-i}|\lambda|^{k-d+i-1}\|Q_{k}\|_2 \qquad& \text{($k\geq j+1$)}\\
    &\leq 
	|\lambda|^j \sum_{k=d-i+1}^d \left|\frac{\lambda}{U}\right|^{k-d+i-1} 
	\qquad& \text{($|\lambda|\leq U$)}\\
	&\leq 
	i|\lambda|^j.
\end{alignat*}
For $i=d-j+1,\hdots,d$, we have
\begin{alignat*}{2}
    \|\mathcal{G}_{i}(\lambda)\|_2 
    &\leq  
    |\lambda|^j U \sum_{k=0}^{d-i} |\lambda|^{k+i-d-1}\|Q_{k}\|_2 \qquad& \text{($-d+i-1+k\leq -1$; $|\lambda|\geq 1$)}\\
        &\leq  
    |\lambda|^j U \sum_{k=0}^{d-i}\|Q_{k}\|_2 \qquad& \text{($k<j$)}\\
    &\leq  
    |\lambda|^j U \sum_{k=0}^{d-i} U^{-(j-k)} \qquad& \text{($k-j\leq -1$; $|U^{-1}|\leq 1$)}\\
    &\leq
    (d-i+1)|\lambda|^j 
\end{alignat*}
Hence
\[
    \|\mathcal G_{i}(\lambda)\|_2
    \leq
    |\lambda|^j\times
    \begin{cases}
       i,&2\leq i \leq d-j,\\
       (d-i+1),&d-j+1\leq i \leq d.
    \end{cases}
\]
It follows that
\begin{equation}\label{eq:bound norm G}
\begin{aligned}
    \|\mathcal G(\lambda)\|_2^2
    &\leq \|\mathcal G_{1}(\lambda)\|_2^2
       +\sum_{i=2}^{d-j}\|\mathcal G_{i}(\lambda)\|_2^2
       +\sum_{i=d-j+1}^{d}\|\mathcal G_{i}(\lambda)\|_2^2\\
    &\leq |\lambda|^{2j}\left(
       1+\sum_{i=2}^{d-j}i^2
        +\sum_{i=d-j+1}^{d}(d-i+1)^2
       \right)\\
    &=|\lambda|^{2j}\left(
       1+\sum_{i=2}^{d-j}i^2
        +\sum_{i=1}^{j}i^2
       \right)\\
    &\leq |\lambda|^{2j}\left(1+\sum_{i=1}^{d-1}i^2\right)\\
    &=B_d^2|\lambda|^{2j}.
\end{aligned}
\end{equation}
The penultimate inequality follows from  $1\le j\le d-1$, and
\[
    \sum_{i=1}^{d-1}i^2
       -\sum_{i=2}^{d-j}i^2
       -\sum_{i=1}^{j}i^2
    =(j-1)(d-j-1)(d+1)\geq0.
\]

The bounds now follow.
\end{proof}

We next bound the norm ratios between exact pencil eigenvectors and their recovered polynomial eigenvectors.

\begin{theorem}
\label{thm:scaled-companion-eigenvector-ratios}
Under the assumptions of Theorem~\ref{thm:scaled-companion-amplification}, let $\lambda$ be an eigenvalue of $S_j^Q$ satisfying $1\leq|\lambda|\leq U$.

Let $z,w\in\mathbb C^{dn}$ be nonzero exact right and left eigenvectors of $S_j^Q(\lambda)$, respectively, and define
\[
x=(e_{d-j}^T\otimes I_n)z,
\qquad
y=(e_1^T\otimes I_n)w.
\]
Then $x$ and $y$ are nonzero, and
\[
\frac{\|z\|_2}{\|x\|_2}\leq\sqrt d,
\qquad
\frac{\|w\|_2}{\|y\|_2}\leq B_d.
\]
\end{theorem}

\begin{proof}
We  use the eigenvector formulas in Proposition~\ref{prop:companion-eigenvectors} to bound the recovery norm ratios for $1\leq|\lambda|\leq U$.
As in the proof of Theorem~\ref{thm:scaled-companion-amplification}, we omit superscripts and subscripts.

Suppose $S(\lambda)z=0$. 
Since $D_R(U)z$ is a right eigenvector of $C(\lambda)$, we have
\[
    z=D_R(U)^{-1}F(\lambda)x_0
    =\mathcal{F}(\lambda)x_0,
\]
for some $x_0\neq 0$ satisfying $Q(\lambda)x_0=0$.
Since the $(d-j)$th block of $D_R(U)$ is $I_n$, the recovered vector is
\[
    x=(e_{d-j}^T\otimes I_n)z=\lambda^j x_0.
\]
Therefore, using the bound~\eqref{eq:bound norm F}, we obtain
\[
    \frac{\|z\|_2}{\|x\|_2}
    =\frac{\|\mathcal{F}(\lambda)x_0\|_2}
           {|\lambda|^j\|x_0\|_2}
    \leq \frac{\|\mathcal{F}(\lambda)\|_2 \|x_0\|_2}
           {|\lambda|^j\|x_0\|_2}         
    \leq\sqrt{d}.
\]

Suppose $w^*S(\lambda)=0$. 
Since $w^*D_L(U)$ is a left eigenvector of $C(\lambda)$, we have
\[
    w^*=y_0^*G(\lambda)D_L(U)^{-1}=y_0^*\mathcal{G}(\lambda),
\]
where $y_0\neq 0$ satisfies $y_0^*Q(\lambda)=0$.
Since the first diagonal block of $D_L(U)$ is $I_n$ and the first block of $G(\lambda)$ is $\lambda^j I_n$, the recovered vector is
\[
    y = (e_1^T\otimes I_n)w=\overline{\lambda}^j y_0.
\]
Therefore, using the bound on $\|\mathcal{G}(\lambda)\|_2$ in \eqref{eq:bound norm G}, we obtain
\[
\frac{\|w\|_2}{\|y\|_2}
    =\frac{\|\mathcal{G}(\lambda)^*y_0\|_2}
           {|\lambda|^j\|y_0\|_2}
    \leq
    \frac{\|\mathcal{G}(\lambda)\|_2\|y_0\|_2}
         {|\lambda|^j\|y_0\|_2}
    \leq B_d.
\]
\end{proof}

We have established the bounds for $S_j^Q$ on $1\leq|\lambda|\leq U$.
We now use polynomial and pencil reversal to obtain corresponding bounds on
$U^{-1}\leq|\lambda|\leq1$.

Let $Q$ satisfy \eqref{eq:interior-coefficient-assumption}, and define
\[
\widetilde Q(\lambda)=\operatorname{rev}_d Q(\lambda)
=\sum_{k=0}^d Q_{d-k}\lambda^k.
\]
Since
\[
\|\widetilde Q_{d-j}\|_2=1,
\qquad
\|\widetilde Q_k\|_2
\leq U^{-|k-(d-j)|},
\qquad 0\leq k\leq d,
\]
the reversed polynomial satisfies the same coefficient bounds, with $j$ replaced by $d-j$.

We can therefore apply the preceding theorems to the scaled companion pencil
\[
S_{d-j}^{\widetilde Q}(\lambda)
=
D_{L,d-j}(U)\,
C_{\widetilde Q}(\lambda)\,
D_{R,d-j}(U).
\]
Its one-sided factorizations are
\[
\begin{aligned}
\left(
G_j^{\widetilde Q}(\lambda)
D_{L,d-j}(U)^{-1}
\right)
S_{d-j}^{\widetilde Q}(\lambda)
&=
\widetilde Q(\lambda)(e_j^T\otimes I_n),\\
S_{d-j}^{\widetilde Q}(\lambda)
\left(
D_{R,d-j}(U)^{-1}F(\lambda)
\right)
&=
(e_1\otimes I_n)\widetilde Q(\lambda).
\end{aligned}
\]

To obtain a linearization of $Q$, we reverse this pencil:
\[
T_j^Q(\lambda)
=
\operatorname{rev}_1 S_{d-j}^{\widetilde Q}(\lambda).
\]
The two one-sided factors are matrix polynomials of degree at most $d-1$.
Reversing them with respect to grade $d-1$ gives
\begin{equation}
\label{eq:factorization-reversal}
\begin{aligned}
&
\left(
\operatorname{rev}_{d-1}G_j^{\widetilde Q}(\lambda)
D_{L,d-j}(U)^{-1}
\right)
T_j^Q(\lambda)
=
Q(\lambda)(e_j^T\otimes I_n),\\[1mm]
&
T_j^Q(\lambda)
\left(
D_{R,d-j}(U)^{-1}
\operatorname{rev}_{d-1}F(\lambda)
\right)
=
(e_1\otimes I_n)Q(\lambda).
\end{aligned}
\end{equation}

Thus, right eigenvectors of $Q$ can be recovered from block $j$ of a right eigenvector of $T_j^Q$, and left eigenvectors from block $1$ of a left eigenvector.
Moreover, reversal maps the region $1\leq|\lambda|\leq U$ to $U^{-1}\leq|\lambda|\leq1$.

The next two theorems establish the corresponding backward-error amplification bounds and recovery norm bounds for this region.

\begin{theorem}\label{thm:reversed-scaled-companion-amplification}
Let \(Q(\lambda)=\sum_{k=0}^{d}Q_k\lambda^k\) be a regular matrix polynomial of degree \(d\geq2\), with \(Q_0\neq0\).
Suppose that \(j\) and \(U\) satisfy \eqref{eq:interior-coefficient-assumption},  and consider the reversed scaled companion pencil $T_j^Q$.

For $U^{-1}\leq|\lambda|\leq1$, the factors
\begin{align*}
&g_{\rm R}^{Q,-}(\lambda)
=
\frac{
W_{T_j^Q}(|\lambda|)
\left\|
\left(\operatorname{rev}_{d-1}G_j^{\widetilde Q}\right)
(\lambda)D_{L,d-j}(U)^{-1}
\right\|_2
}{
W_Q(|\lambda|)
},\qquad \text{and} \\
&g_{\rm L}^{Q,-}(\lambda)
=
\frac{
W_{T_j^Q}(|\lambda|)
\left\|
D_{R,d-j}(U)^{-1}
\left(\operatorname{rev}_{d-1}F\right)(\lambda)
\right\|_2
}{
W_Q(|\lambda|)
}
\end{align*}
satisfy
\[
g_{\rm R}^{Q,-}(\lambda)\leq A_dB_d,
\qquad
g_{\rm L}^{Q,-}(\lambda)\leq A_d\sqrt d.
\]
\end{theorem}

\begin{proof}
Let \(\widetilde Q=\operatorname{rev}_d Q\).
Since \(\widetilde Q_d=Q_0\neq0\) and \(\widetilde Q_0=Q_d\neq0\), the polynomial \(\widetilde Q\) has degree \(d\) and a nonzero constant coefficient. 
It is regular, and its coefficients satisfy
\[
\|\widetilde Q_{d-j}\|_2=1,
\qquad
\|\widetilde Q_k\|_2
=\|Q_{d-k}\|_2
\leq U^{-|k-(d-j)|},
\qquad 0\leq k\leq d.
\]
Thus, \(\widetilde Q\) satisfies the hypotheses of Theorem~\ref{thm:scaled-companion-amplification} with index \(d-j\).

Let $\mu=\lambda^{-1}$, so that
$1\leq|\mu|\leq U$.
Since
\[
T_j^Q(\lambda)
=\lambda S_{d-j}^{\widetilde Q}(\mu),
\]
reversal gives
\[
W_{T_j^Q}(|\lambda|)
=|\lambda|W_{S_{d-j}^{\widetilde Q}}(|\mu|),
\qquad
W_Q(|\lambda|)
=|\lambda|^dW_{\widetilde Q}(|\mu|).
\]
Moreover,
\[
\begin{aligned}
\left(\operatorname{rev}_{d-1}
G_j^{\widetilde Q}\right)(\lambda)
D_{L,d-j}(U)^{-1}
&=
\lambda^{d-1}
G_j^{\widetilde Q}(\mu)
D_{L,d-j}(U)^{-1},\\
D_{R,d-j}(U)^{-1}
\left(\operatorname{rev}_{d-1}F\right)(\lambda)
&=
\lambda^{d-1}
D_{R,d-j}(U)^{-1}F(\mu).
\end{aligned}
\]
The powers of $|\lambda|$ therefore cancel, giving
\[
g_{\rm R}^{Q,-}(\lambda)
=g_{\rm R}^{\widetilde Q,+}(\mu),
\qquad
g_{\rm L}^{Q,-}(\lambda)
=g_{\rm L}^{\widetilde Q,+}(\mu).
\]
The result follows from Theorem~\ref{thm:scaled-companion-amplification}, applied to $\widetilde Q$ with index $d-j$.
\end{proof}

We next bound the norm ratios between exact eigenvectors of the reversed pencil and their recovered polynomial eigenvectors.

\begin{theorem}
\label{thm:reversed-scaled-companion-eigenvector-ratios}
Under the assumptions of Theorem~\ref{thm:reversed-scaled-companion-amplification}, let $\lambda$ be an eigenvalue of $T_j^Q$ satisfying $U^{-1}\leq|\lambda|\leq1$.

Let $z,w\in\mathbb C^{dn}$ be nonzero exact right and left eigenvectors of $T_j^Q(\lambda)$, respectively, and define
\[
x=(e_j^T\otimes I_n)z,
\qquad
y=(e_1^T\otimes I_n)w.
\]
Then $x$ and $y$ are nonzero, and
\[
\frac{\|z\|_2}{\|x\|_2}\leq\sqrt d,
\qquad
\frac{\|w\|_2}{\|y\|_2}\leq B_d.
\]
\end{theorem}

\begin{proof}
Let $\mu=\lambda^{-1}$. Since
\[
T_j^Q(\lambda)
=\lambda S_{d-j}^{\widetilde Q}(\mu),
\]
the two pencils have the same right and left nullspaces.
They also use the same recovery blocks: block $j$ for right eigenvectors and block $1$ for left eigenvectors.
The bounds therefore follow directly from Theorem~\ref{thm:scaled-companion-eigenvector-ratios}, applied to $\widetilde Q$ with index $d-j$.
\end{proof}

\subsection{Endpoint regions}
\label{sec:endpoint-extension}

Assume $s\geq2$.
The preceding theorems cover the regions between consecutive tropical roots.
We now consider the two remaining regions, $0<|\lambda|\leq\alpha_1$ and $|\lambda|\geq\alpha_s$.
Under the corresponding tropical scalings, these become $0<|\lambda|\leq U^{-1}$ and $|\lambda|\geq U$, respectively, where $\lambda$ now denotes the scaled variable.
Thus, we extend our analysis beyond the interval $U^{-1}\leq|\lambda|\leq U$ considered in the preceding subsection.

We use the same scaled companion pencils as in the adjacent interior intervals, but recover right eigenvectors from block $1$ instead of the interior recovery block.
Left eigenvectors are still recovered from block $1$.

To establish the bounds, we again work with a general polynomial $Q(\lambda)=\sum_{k=0}^d Q_k\lambda^k$ satisfying \eqref{eq:interior-coefficient-assumption} for some $1\leq j\leq d-1$ and $U>1$.
The endpoint regions require one additional condition:
\[
\|Q_d\|_2=U^{-(d-j)}
\quad\text{for the outer region},
\qquad
\|Q_0\|_2=U^{-j}
\quad\text{for the inner region}.
\]
These conditions hold for the scaled polynomials associated with the last and first tropical roots, respectively.

We first consider the outer region $|\lambda|\geq U$.
We use the scaled companion pencil
\[
S_j^Q(\lambda)
=
D_{L,j}(U)C_Q(\lambda)D_{R,j}(U).
\]
Since the first diagonal block of $D_{R,j}(U)$ is $U^{d-j-1}I_n$, the companion one-sided factorization
\[
G_1^Q(\lambda)C_Q(\lambda)
=
Q(\lambda)(e_1^T\otimes I_n)
\]
gives
\[
\left(
U^{-(d-j-1)}
G_1^Q(\lambda)D_{L,j}(U)^{-1}
\right)
S_j^Q(\lambda)
=
Q(\lambda)(e_1^T\otimes I_n).
\]
The left factorization remains unchanged:
\[
S_j^Q(\lambda)
\left(D_{R,j}(U)^{-1}F(\lambda)\right)
=
(e_1\otimes I_n)Q(\lambda).
\]

Next, consider the inner region $0<|\lambda|\leq U^{-1}$.
We use the reversed scaled companion pencil
\[
T_j^Q(\lambda)
=
\operatorname{rev}_1 S_{d-j}^{\widetilde Q}(\lambda),
\qquad
\widetilde Q=\operatorname{rev}_d Q.
\]
The first diagonal block of $D_{R,d-j}(U)$ is $U^{j-1}I_n$.
Applying the preceding companion factorization to $\widetilde Q$ and then reversing gives
\[
\begin{aligned}
&
\left(
U^{-(j-1)}
\left(\operatorname{rev}_{d-1}
G_1^{\widetilde Q}\right)(\lambda)
D_{L,d-j}(U)^{-1}
\right)
T_j^Q(\lambda)\\
&\hspace{3cm}
=
Q(\lambda)(e_1^T\otimes I_n).
\end{aligned}
\]
The left factorization is
\[
T_j^Q(\lambda)
\left(
D_{R,d-j}(U)^{-1}
\left(\operatorname{rev}_{d-1}F\right)(\lambda)
\right)
=
(e_1\otimes I_n)Q(\lambda).
\]

Thus, in both endpoint regions, we recover right and left eigenvectors from the first block:
\[
x=(e_1^T\otimes I_n)z,
\qquad
y=(e_1^T\otimes I_n)w.
\]
We next establish the backward-error amplification bounds for these recovery rules, followed by the corresponding norm bounds for exact pencil eigenvectors.

\begin{theorem}
\label{thm:endpoint-extension}
Let \(Q(\lambda)=\sum_{k=0}^{d}Q_k\lambda^k\) be a regular matrix polynomial of degree \(d\geq2\), with \(Q_0\neq0\).
Suppose that \(j\) and \(U\) satisfy \eqref{eq:interior-coefficient-assumption}, and set
\[
A_d=\sqrt{d+1}+1,
\qquad
B_d=\left(1+\frac{(d-1)d(2d-1)}{6}\right)^{1/2}.
\]

\emph{(a) Upper endpoint.}
Assume additionally that $\|Q_d\|_2=U^{-(d-j)}$.
Consider the scaled companion pencil $S_j^Q(\lambda)$.
For $|\lambda|\geq U$, the factors
\[
\begin{aligned}
g_{\rm R}^{Q,+}(\lambda)
&=
\frac{
W_{S_j^Q}(|\lambda|)
U^{-(d-j-1)}
\|G_1^Q(\lambda)D_{L,j}(U)^{-1}\|_2
}{
W_Q(|\lambda|)
},\\
g_{\rm L}^{Q,+}(\lambda)
&=
\frac{
W_{S_j^Q}(|\lambda|)
\|D_{R,j}(U)^{-1}F(\lambda)\|_2
}{
W_Q(|\lambda|)
}
\end{aligned}
\]
satisfy
\[
g_{\rm R}^{Q,+}(\lambda)\leq A_dB_d,
\qquad
g_{\rm L}^{Q,+}(\lambda)\leq A_d\sqrt d.
\]

\emph{(b) Lower endpoint.}
Assume additionally that $\|Q_0\|_2=U^{-j}$.
Let $\widetilde Q=\operatorname{rev}_d Q$, and consider
the reversed scaled companion pencil
\[
T_j^Q(\lambda)
=\operatorname{rev}_1 S_{d-j}^{\widetilde Q}(\lambda).
\]
For $0<|\lambda|\leq U^{-1}$, the factors
\[
\begin{aligned}
g_{\rm R}^{Q,-}(\lambda)
&=
\frac{
W_{T_j^Q}(|\lambda|)
U^{-(j-1)}
\left\|
\left(\operatorname{rev}_{d-1}
G_1^{\widetilde Q}\right)(\lambda)
D_{L,d-j}(U)^{-1}
\right\|_2
}{
W_Q(|\lambda|)
},\\
g_{\rm L}^{Q,-}(\lambda)
&=
\frac{
W_{T_j^Q}(|\lambda|)
\left\|
D_{R,d-j}(U)^{-1}
\left(\operatorname{rev}_{d-1}F\right)(\lambda)
\right\|_2
}{
W_Q(|\lambda|)
}
\end{aligned}
\]
satisfy
\[
g_{\rm R}^{Q,-}(\lambda)\leq A_dB_d,
\qquad
g_{\rm L}^{Q,-}(\lambda)\leq A_d\sqrt d.
\]
\end{theorem}

\begin{proof}
We first prove part (a).
As in the preceding proofs, we simplify the notation by writing
\[
S=S_j^Q,\qquad
D_L=D_{L,j}(U),\qquad
D_R=D_{R,j}(U),\qquad
r=|\lambda|.
\]
By Proposition~\ref{propo:scaled-companion-coefficient-bounds},
\[
\|S_0\|_2\leq\sqrt{d+1},
\qquad
\|S_1\|_2\leq U^{-1}.
\]
Since $r\geq U$, we obtain
\[
\begin{aligned}
W_S(r)
&\leq\sqrt{d+1}+\frac rU\\
&\leq\left(\sqrt{d+1}+1\right)\frac rU
=A_d\frac rU.
\end{aligned}
\]
Moreover, the assumption on $Q_d$ gives
\[
W_Q(r)\geq\|Q_d\|_2r^d
=U^{-(d-j)}r^d.
\]

We next bound the norms of the one-sided factors.
For the right factor, write
\[
G_1^Q(\lambda)D_L^{-1}
=\lambda^{d-1}
[H_1(\lambda),\ldots,H_d(\lambda)].
\]
The companion formula gives $H_1(\lambda)=I_n$ and, for $2\leq i\leq d$,
\[
H_i(\lambda)
=c_i\sum_{k=0}^{d-i}
\lambda^{k+i-d-1}Q_k,
\]
where
\[
|c_i|=
\begin{cases}
U^{d-j+1-i},&2\leq i\leq d-j,\\
U,&d-j+1\leq i\leq d.
\end{cases}
\]
The first range is empty when $d-j=1$.

Since $r\geq U$, for $2\leq i\leq d-j$,
\[
\begin{aligned}
\|H_i(\lambda)\|_2
&\leq
\sum_{k=0}^{d-i}
U^{d-j+1-i+k+i-d-1}\|Q_k\|_2\\
&\leq
\sum_{k=0}^{d-i}U^{k-j-|k-j|}\\
&\leq d-i+1,
\end{aligned}
\]
where the last inequality follows from
$k-j-|k-j|\leq0$.

For $d-j+1\leq i\leq d$, we similarly obtain
\[
\begin{aligned}
\|H_i(\lambda)\|_2
&\leq
\sum_{k=0}^{d-i}
U^{k+i-d}\|Q_k\|_2\\
&\leq
\sum_{k=0}^{d-i}
U^{k+i-d-|k-j|}\\
&\leq d-i+1.
\end{aligned}
\]
Indeed, $k\leq d-i\leq j-1$, so $k<j$ and
\[
k+i-d-|k-j|
=2k+i-d-j
\leq k-j<0.
\]

Consequently,
\[
\begin{aligned}
\|G_1^Q(\lambda)D_L^{-1}\|_2^2
&\leq
r^{2(d-1)}
\left(1+\sum_{i=2}^d(d-i+1)^2\right)\\
&=B_d^2r^{2(d-1)}.
\end{aligned}
\]

For the left factor, the definition of $D_R$ gives
\[
D_R^{-1}F(\lambda)
=
U^{-(d-j-1)}\lambda^{d-1}
\begin{bmatrix}
I_n\\
(U/\lambda)I_n\\
\vdots\\
(U/\lambda)^{d-j-1}I_n\\
(U/\lambda)^{d-j-1}\lambda^{-1}I_n\\
\vdots\\
(U/\lambda)^{d-j-1}\lambda^{-j}I_n
\end{bmatrix}.
\]
Since $r\geq U>1$, every block inside the brackets has norm at most $1$.
Hence
\[
\|D_R^{-1}F(\lambda)\|_2
\leq\sqrt d\,U^{-(d-j-1)}r^{d-1}.
\]

Combining these estimates yields
\[
\begin{aligned}
g_{\rm R}^{Q,+}(\lambda)
&\leq
\frac{
A_d(r/U)\,U^{-(d-j-1)}B_dr^{d-1}
}{
U^{-(d-j)}r^d
}\\
&=A_dB_d,
\end{aligned}
\]
and
\[
\begin{aligned}
g_{\rm L}^{Q,+}(\lambda)
&\leq
\frac{
A_d(r/U)\,\sqrt d\,U^{-(d-j-1)}r^{d-1}
}{
U^{-(d-j)}r^d
}\\
&=A_d\sqrt d.
\end{aligned}
\]
This proves part (a).

For part (b), let $\widetilde Q=\operatorname{rev}_d Q$.
Its coefficients satisfy
\[
\begin{aligned}
\|\widetilde Q_{d-j}\|_2&=1,\\
\|\widetilde Q_k\|_2
&\leq U^{-|k-(d-j)|},
\qquad 0\leq k\leq d,\\
\|\widetilde Q_d\|_2
&=\|Q_0\|_2=U^{-j}.
\end{aligned}
\]
Thus, $\widetilde Q$ satisfies the hypotheses of part (a) with index $d-j$.

Let $0<|\lambda|\leq U^{-1}$ and set
$\mu=\lambda^{-1}$. Then $|\mu|\geq U$, and
\[
T_j^Q(\lambda)
=\lambda S_{d-j}^{\widetilde Q}(\mu).
\]
Reversal gives
\[
W_{T_j^Q}(|\lambda|)
=|\lambda|W_{S_{d-j}^{\widetilde Q}}(|\mu|), \qquad W_Q(|\lambda|)
=|\lambda|^dW_{\widetilde Q}(|\mu|).
\]
Moreover,
\[
\left(\operatorname{rev}_{d-1}
G_1^{\widetilde Q}\right)(\lambda)
D_{L,d-j}(U)^{-1}=
\lambda^{d-1}
G_1^{\widetilde Q}(\mu)
D_{L,d-j}(U)^{-1},
\]
and
\[
D_{R,d-j}(U)^{-1}
\left(\operatorname{rev}_{d-1}F\right)(\lambda)\\
=
\lambda^{d-1}
D_{R,d-j}(U)^{-1}F(\mu).
\]
Consequently, the powers of $|\lambda|$ cancel,
giving
\[
\begin{aligned}
g_{\rm R}^{Q,-}(\lambda)
&=g_{\rm R}^{\widetilde Q,+}(\mu)
\leq A_dB_d,\\
g_{\rm L}^{Q,-}(\lambda)
&=g_{\rm L}^{\widetilde Q,+}(\mu)
\leq A_d\sqrt d.
\end{aligned}
\]
Both inequalities follow from part (a).
\end{proof}

We next bound the norm ratios between exact pencil eigenvectors and their recovered polynomial eigenvectors.

\begin{theorem}
\label{thm:endpoint-eigenvector-ratios}
Let \(Q(\lambda)=\sum_{k=0}^{d}Q_k\lambda^k\) be a regular matrix polynomial of degree \(d\geq2\), with \(Q_0\neq0\).
Suppose that \(j\) and \(U\) satisfy \eqref{eq:interior-coefficient-assumption}.

\emph{(a) Upper endpoint.}
Let $\lambda$ be an eigenvalue of $S_j^Q$ satisfying $|\lambda|\geq U$.
For nonzero exact right and left eigenvectors $z$ and $w$ of $S_j^Q(\lambda)$, respectively, define
\[
x=(e_1^T\otimes I_n)z,
\qquad
y=(e_1^T\otimes I_n)w.
\]
Then $x$ and $y$ are nonzero, and
\[
\frac{\|z\|_2}{\|x\|_2}\leq\sqrt d,
\qquad
\frac{\|w\|_2}{\|y\|_2}\leq B_d.
\]

\emph{(b) Lower endpoint.}
Let $\lambda$ be an eigenvalue of $T_j^Q$ satisfying $0<|\lambda|\leq U^{-1}$.
For nonzero exact right and left eigenvectors $z$ and $w$ of $T_j^Q(\lambda)$, respectively, define
\[
x=(e_1^T\otimes I_n)z,
\qquad
y=(e_1^T\otimes I_n)w.
\]
Then $x$ and $y$ are nonzero, and
\[
\frac{\|z\|_2}{\|x\|_2}\leq\sqrt d,
\qquad
\frac{\|w\|_2}{\|y\|_2}\leq B_d.
\]
\end{theorem}

\begin{proof}
We first prove part (a).
Write
\[
S=S_j^Q,\qquad
D_L=D_{L,j}(U),\qquad
D_R=D_{R,j}(U),\qquad
r=|\lambda|.
\]

Suppose $S(\lambda)z=0$.
By Proposition~\ref{prop:companion-eigenvectors},
\[
z=D_R^{-1}F(\lambda)x_0
\]
for some nonzero $x_0$ satisfying $Q(\lambda)x_0=0$.
The recovered vector is
\[
x=(e_1^T\otimes I_n)z
=U^{-(d-j-1)}\lambda^{d-1}x_0.
\]
Since $r\geq U>1$, every block of $D_R^{-1}F(\lambda)$ has norm at most $U^{-(d-j-1)}r^{d-1}$, as shown in the proof of Theorem~\ref{thm:endpoint-extension}.
Therefore, $x\ne0$ and
\[
\frac{\|z\|_2}{\|x\|_2}
=
\frac{\|D_R^{-1}F(\lambda)x_0\|_2}
{U^{-(d-j-1)}r^{d-1}\|x_0\|_2}
\leq\sqrt d.
\]

Now suppose $w^*S(\lambda)=0$.
The companion eigenvector formula gives
\[
w^*=y_0^*G_1^Q(\lambda)D_L^{-1}
\]
for some nonzero $y_0$ satisfying $y_0^*Q(\lambda)=0$.
Since the first block of $G_1^Q(\lambda)D_L^{-1}$ is $\lambda^{d-1}I_n$, the recovered vector satisfies
\[
y^*=\lambda^{d-1}y_0^*.
\]
Thus, $y\ne0$ and
\[
\frac{\|w\|_2}{\|y\|_2}
\leq
\frac{\|G_1^Q(\lambda)D_L^{-1}\|_2}
{r^{d-1}}
\leq B_d,
\]
where the last inequality follows from the proof of Theorem~\ref{thm:endpoint-extension}.
This proves part (a).

For part (b), let \(\widetilde Q=\operatorname{rev}_d Q\).
Reversal preserves regularity and the nonzero endpoint coefficients, and transforms \eqref{eq:interior-coefficient-assumption} into the same coefficient bounds with index \(d-j\).
Thus, \(\widetilde Q\) satisfies the hypotheses of part (a) with this index.


Let $\mu=\lambda^{-1}$.
Then $|\mu|\geq U$, and
\[
T_j^Q(\lambda)
=\lambda S_{d-j}^{\widetilde Q}(\mu).
\]
The two pencils have the same right and left nullspaces and use the same recovery block, namely block $1$.
The bounds therefore follow directly from part (a), applied to $\widetilde Q$ with index $d-j$.
\end{proof}

Theorems~\ref{thm:endpoint-extension} and~\ref{thm:endpoint-eigenvector-ratios} apply to the scaled polynomials \(Q_{v_{s-1}}\) and \(Q_{v_1}\) in the upper and lower endpoint regions, respectively. Indeed, Proposition~\ref{propo:tropical-coefficient-decay} gives the required coefficient bounds, together with
\[
\begin{aligned}
\|Q_{v_{s-1},d}\|_2
&=U_{v_{s-1}}^{-(d-v_{s-1})},\\
\|Q_{v_1,0}\|_2
&=U_{v_1}^{-v_1}.
\end{aligned}
\]
Thus, both the backward-error amplification bounds and the exact eigenvector recovery bounds hold in the two endpoint regions.

\subsection{A single tropical root}
\label{sec:single-root-pencil}

Suppose \(s=1\). 
The upper hull is the segment joining indices \(0\) and \(d\), so the unique tropical root is
\[
    \alpha=\left(\frac{p_0}{p_d}\right)^{1/d},
    \qquad p_k=\|P_k\|_2.
\]
Scale \(P\) by setting
\[
    Q(\lambda)=\frac{P(\alpha\lambda)}{p_0}
    =\sum_{k=0}^d Q_k\lambda^k.
\]
Then
\[
    \|Q_0\|_2=\|Q_d\|_2=1,
    \qquad
    \|Q_k\|_2\leq1,\quad 0\leq k\leq d.
\]

We use the Frobenius companion pencil \(C_Q(\lambda)\).
For \(|\lambda|\leq1\), we recover right eigenvectors from block \(d\). 
For \(|\lambda|\geq1\), we recover them from block \(1\).
 Left recovery remains from block \(1\).
At \(|\lambda|=1\), either right recovery block may be used.
Thus, a single pencil covers the whole complex plane.

The corresponding one-sided factorizations are
\[
\begin{aligned}
G_d^Q(\lambda)C_Q(\lambda)
    &=Q(\lambda)(e_d^T\otimes I_n),\\
G_1^Q(\lambda)C_Q(\lambda)
    &=Q(\lambda)(e_1^T\otimes I_n),\\
C_Q(\lambda)F(\lambda)
    &=(e_1\otimes I_n)Q(\lambda).
\end{aligned}
\]

We give the residual-amplification and recovery norm ratio bounds for completeness \cite{GaubertSharify2009,HighamLiTisseur2007}.


\begin{theorem}
\label{prop:single-root-amplification}
Let $Q(\lambda)=\sum_{k=0}^d Q_k \lambda^k$ be a matrix polynomial of degree $d$ such that
\[
    \|Q_0\|_2=\|Q_d\|_2=1,
    \qquad
    \|Q_k\|_2\leq1,\quad 0\leq k\leq d.
\]
Let
\[
    A_d=\sqrt{d+1}+1,
    \qquad
    \widehat B_d=
    \left(\frac{d(d+1)(2d+1)}{6}\right)^{1/2}.
\]
For \(0\leq|\lambda|\leq1\), define
\[
g_{\rm R}(\lambda)
=
\frac{
W_{C_Q}(|\lambda|)\|G_d^Q(\lambda)\|_2
}{
W_Q(|\lambda|)
},\qquad 
g_{\rm L}(\lambda)
=
\frac{
W_{C_Q}(|\lambda|)\|F(\lambda)\|_2
}{
W_Q(|\lambda|)
}.
\]
For \(|\lambda|\geq1\), use the same definitions with \(G_d^Q\) replaced by \(G_1^Q\).

Then
\[
    g_{\rm R}(\lambda)\leq A_d\widehat B_d,
    \qquad
    g_{\rm L}(\lambda)\leq A_d\sqrt d.
\]

For exact right and left eigenvectors \(z\) and \(w\) of \(C_Q(\lambda)\), respectively, recover
\[
    x=
    \begin{cases}
    (e_d^T\otimes I_n)z,&|\lambda|\leq1,\\
    (e_1^T\otimes I_n)z,&|\lambda|\geq1,
    \end{cases}
    \qquad
    y=(e_1^T\otimes I_n)w.
\]
Their recovery norm ratios satisfy
\[
    \frac{\|z\|_2}{\|x\|_2}\leq\sqrt d,
    \qquad
    \frac{\|w\|_2}{\|y\|_2}\leq\widehat B_d.
\]
\end{theorem}

\begin{proof}
Write \(r=|\lambda|\). Since
\(\|Q_0\|_2=\|Q_d\|_2=1\),
\[
    W_Q(r)\geq1+r^d.
\]
Write $C_Q(\lambda)=\lambda C_1 + C_0$. 
The block structure of the companion pencil gives
\[
    \|C_0\|_2\leq\sqrt{d+1},
    \qquad
    \|C_1\|_2\leq1.
\]
Hence
\[
W_{C_Q}(r)
\leq\sqrt{d+1}+r
\leq
\begin{cases}
A_d,&0\leq r\leq1,\\
A_dr,&r\geq1.
\end{cases}
\]

First, consider \(0\leq r\leq1\).
We recover right eigenvectors from block \(d\).
The companion factor \(G_d^Q\) has block entries
\[
    G_d^Q(\lambda)
    =[I_n,G_2(\lambda),\ldots,G_d(\lambda)],
\]
where
\[
    G_i(\lambda)
    =\sum_{k=d-i+1}^d
       Q_k\lambda^{k-d+i-1},
    \qquad 2\leq i\leq d.
\]
Each sum contains \(i\) terms, each of norm at most one.
Therefore,
\[
    \|G_i(\lambda)\|_2\leq i,
    \qquad 2\leq i\leq d.
\]
Consequently,
\[
\|G_d^Q(\lambda)\|_2^2
\leq1+\sum_{i=2}^d i^2
=\widehat B_d^2.
\]

Moreover,
$
    \|F(\lambda)\|_2^2
    =\sum_{k=0}^{d-1}r^{2k}\leq d.
$
Since \(W_Q(r)\geq1\), we obtain
\[
g_{\rm R}(\lambda)
\leq
\frac{A_d\widehat B_d}{W_Q(r)}
\leq A_d\widehat B_d,\qquad
g_{\rm L}(\lambda)
\leq
\frac{A_d\sqrt d}{W_Q(r)}
\leq A_d\sqrt d.
\]

Next, consider \(r\geq1\).
We recover right eigenvectors from block \(1\).
The companion factor \(G_1^Q\) has block entries
\[
    G_1^Q(\lambda)
    =\lambda^{d-1}
    [I_n,G_2(\lambda),\ldots,G_d(\lambda)],
\]
where
\[
    G_i(\lambda)
    =-\sum_{k=0}^{d-i}
       Q_k\lambda^{k+i-d-1},
    \qquad 2\leq i\leq d.
\]
Each sum contains \(d-i+1\) terms, each of norm at most one. 
Hence
\[
    \|G_i(\lambda)\|_2\leq d-i+1,
    \qquad 2\leq i\leq d.
\]
It follows that
\[
\|G_1^Q(\lambda)\|_2^2
\leq r^{2(d-1)}
\left(1+\sum_{i=2}^d(d-i+1)^2\right)
\leq\widehat B_d^2r^{2(d-1)}.
\]
Similarly,
\[
\|F(\lambda)\|_2^2
=\sum_{k=0}^{d-1}r^{2k}
\leq dr^{2(d-1)}.
\]
Since \(W_Q(r)\geq r^d\), we obtain
\[
g_{\rm R}(\lambda)
\leq
\frac{A_dr\,\widehat B_dr^{d-1}}{r^d}
=A_d\widehat B_d,\qquad 
g_{\rm L}(\lambda)
\leq
\frac{A_dr\,\sqrt d\,r^{d-1}}{r^d}
=A_d\sqrt d.
\]

We now bound the recovery norm ratios for exact
eigenvectors.

Suppose \(C_Q(\lambda)z=0\).
By Proposition~\ref{prop:companion-eigenvectors},
\[
    z=F(\lambda)x_0
\]
for some nonzero \(x_0\) satisfying \(Q(\lambda)x_0=0\).

For \(0\leq r\leq1\), recovery from block \(d\) gives \(x=x_0\). 
Thus,
\[
\frac{\|z\|_2^2}{\|x\|_2^2}
=\frac{\|F(\lambda)x_0\|_2^2}{\|x_0\|_2^2}
=\sum_{k=0}^{d-1}r^{2k}
\leq d.
\]
For \(r\geq1\), recovery from block \(1\) gives \(x=\lambda^{d-1}x_0\), and hence
\[
\frac{\|z\|_2^2}{\|x\|_2^2}
=\frac{\|F(\lambda)x_0\|_2^2}
        {r^{2(d-1)}\|x_0\|_2^2}
=\sum_{k=0}^{d-1}r^{-2k}
\leq d.
\]

Finally, suppose \(w^*C_Q(\lambda)=0\), and let \(y=(e_1^T\otimes I_n)w\).
For \(0\leq r\leq1\), the companion eigenvector formula gives
\[
    w^*=y^*G_d^Q(\lambda).
\]
Therefore,
\[
\frac{\|w\|_2}{\|y\|_2}
\leq\|G_d^Q(\lambda)\|_2
\leq\widehat B_d.
\]
For \(r\geq1\), the corresponding formula is
\[
    w^*=\lambda^{-(d-1)}y^*G_1^Q(\lambda).
\]
Consequently,
\[
\frac{\|w\|_2}{\|y\|_2}
\leq r^{-(d-1)}\|G_1^Q(\lambda)\|_2
\leq\widehat B_d.
\]
%
\end{proof}

\subsection{Proof of the main theorems}

We now combine the preceding results to prove Theorems~\ref{thm:family-amplification} and~\ref{thm:family-eigenvector-ratios}.

\begin{proof}[Proof of Theorem~\ref{thm:family-amplification}]
Suppose $s\geq2$.
Theorem~\ref{thm:scaled-companion-amplification} gives the backward-error amplification bounds for the interior regions assigned to $L_j^+$, while Theorem~\ref{thm:reversed-scaled-companion-amplification} gives those for the regions assigned to $L_j^-$.
Theorem~\ref{thm:endpoint-extension} covers the two endpoint regions.

The change of variable $\lambda\mapsto\lambda/\gamma_j$ preserves the backward-error amplification factors.
Thus, the bounds also hold for the pencils expressed in the original eigenvalue variable.
Combining these bounds with the one-sided factorizations gives the stated backward-error estimates for arbitrary pencil vectors.

If $s=1$, the result follows from Theorem~\ref{prop:single-root-amplification}, using the change of variable $\lambda\mapsto\lambda/\alpha$.
\end{proof}

\begin{proof}[Proof of Theorem~\ref{thm:family-eigenvector-ratios}]
Suppose $s\geq2$.
Theorems~\ref{thm:scaled-companion-eigenvector-ratios} and~\ref{thm:reversed-scaled-companion-eigenvector-ratios} give the recovery norm bounds for the interior regions, while Theorem~\ref{thm:endpoint-eigenvector-ratios} covers the two endpoint regions.

The change of variable $\lambda\mapsto\lambda/\gamma_j$ does not change the pencil eigenvectors or their recovery blocks.
Hence the same bounds hold for the pencils in the original eigenvalue variable.

If $s=1$, the bounds follow from Theorem~\ref{prop:single-root-amplification}, using the change of variable $\lambda\mapsto\lambda/\alpha$.
\end{proof}

\section{Numerical experiments}

We compare the proposed family of scaled companion linearizations with several existing approaches.
The six methods, identified by their labels in the figures, are:

\begin{itemize}
    \item \emph{Companion}: the Frobenius companion
    pencil of the polynomial normalized so that
    its largest coefficient norm is one.
    
    \item \emph{Classical tropical scaling}:
    the companion pencils obtained by scaling the polynomial at its distinct tropical roots, following Gaubert and Sharify~\cite{GaubertSharify2009}.
For each scaling parameter $\alpha_i$, the computed eigenvalues are sorted by increasing modulus, and those with indices $nv_{i-1}+1,\ldots,nv_i$ are retained, together with their corresponding eigenvectors. 
Thus, $n(v_i-v_{i-1})$ eigenpairs are selected from the $i$th scaled companion pencil.

    \item \emph{Tropically scaled companion}:
    the proposed family of scaled companion
    linearizations.

    \item \emph{Scaled companion (2021)}:
    the tropically scaled companion linearization
    of Van Barel and Tisseur \cite{TisseurVanBarel2021}.

    \item \emph{Scaled Lagrange (classical)} and
    \emph{Scaled Lagrange (separated)}:
    the two variants of the tropically scaled
    Lagrange linearizations of Van Barel and
    Tisseur (2018) \cite{VanBarelTisseur2018}.
\end{itemize}

For the three scaled linearizations of Tisseur and Van Barel, we use a modified version of the QZ algorithm implemented in LAPACK. 
In the standard implementation, sufficiently small diagonal entries of the triangular factor associated with the coefficient of $\lambda$ may be treated as zero, resulting in finite eigenvalues being classified as infinite. 
To avoid this premature classification, we replace the standard tolerance for detecting such entries by the smallest positive representable floating-point number. 
We also increase the maximum number of QZ iterations by a factor of ten to accommodate the potentially slower convergence resulting from this modification.

For the companion pencil and the companion pencils used in classical tropical scaling, we recover right polynomial eigenvectors from the last block of the computed pencil eigenvector if the modulus of the eigenvalue in the scaled variable is at most one, and from the first block otherwise.
For the three linearizations of Tisseur and Van Barel, we recover right polynomial eigenvectors from the block of largest Euclidean norm.
For the proposed family of scaled companion linearizations, we use the recovery rules established in the preceding sections.
In all six methods, we recover left polynomial eigenvectors from the first block of the corresponding pencil eigenvector.

We evaluate the coefficientwise relative polynomial backward errors of the recovered right and left eigenpairs.

We consider three examples. The first is a quartic with prescribed eigenvalue moduli between tropical roots. 
The second is a random quartic with widely varying coefficient norms and ill-conditioned coefficient matrices. The third is the planar waveguide problem from the NLEVP collection \cite{BetckeEtAl2013}.

\subsection{A quartic with prescribed eigenvalue moduli}

We construct a dense quartic matrix polynomial for which we can prescribe the tropical roots and control the location of most of its eigenvalues independently. 
In particular, we aim to generate a polynomial with many eigenvalue moduli far from its tropical roots.

Fix $R>1$ and set
\[
    (a_0,a_1,a_2,a_3,a_4)
    =(1,R^3,R^4,R^3,1).
\]
These are the desired coefficient norms of the matrix polynomial. 
The corresponding tropical polynomial
\[
    \max\{1,R^3r,R^4r^2,R^3r^3,r^4\}
\]
has four distinct tropical roots,
\[
    R^{-3},\quad R^{-1},\quad R,\quad R^3.
\]

We first construct a diagonal matrix polynomial
\[
    D(\lambda)
    =\operatorname{diag}
      (p_0(\lambda),\ldots,p_{n-1}(\lambda))
    =\sum_{k=0}^{4}D_k\lambda^k.
\]
Since the coefficient matrices $D_k$ are diagonal, their spectral norms are the largest magnitudes of their diagonal entries. 
We use the scalar polynomial $p_0$ to enforce the desired coefficient norms and the remaining scalar polynomials to prescribe most of the eigenvalues of $D$.

Specifically, we set
\[
    p_0(\lambda)
    =\sum_{k=0}^{4}a_ke^{-ik\vartheta_0}\lambda^k,
\]
where $\vartheta_0$ is a random angle. 
The rotation introduces a random phase into the roots of $p_0$ without changing its coefficient magnitudes.

For each $\ell=1,\ldots,n-1$, we prescribe four roots
\[
    \zeta_{\ell,m}
    =R^{\xi_{\ell,m}}e^{i\vartheta_{\ell,m}},
    \qquad m=1,\ldots,4,
\]
where the angles $\vartheta_{\ell,m}$ are independent and uniformly distributed on $[-\pi,\pi]$.
We choose the exponents $\xi_{\ell,m}$ near the values $-2$, $0$, and $2$, which correspond to the geometric means of consecutive tropical roots. 
More precisely, we set
\[
    \xi_{\ell,m}=q_{\ell,m}+\delta_{\ell,m},
\]
where the first three values of $q_{\ell,m}$ are $-2$, $0$, and $2$, the fourth cycles through these three values as $\ell$ varies, and the perturbations $\delta_{\ell,m}$ are independent and uniformly distributed on $[-0.4,0.4]$. 
This construction places the prescribed root moduli in the three intervals between consecutive tropical roots, away from their endpoints.

To ensure that these scalar polynomials do not alter the desired coefficient norms, we write
\[
    \prod_{m=1}^{4}(\lambda-\zeta_{\ell,m})
    =\sum_{k=0}^{4}b_{\ell,k}\lambda^k
\]
and define
\[
    p_\ell(\lambda)
    =c_\ell\prod_{m=1}^{4}(\lambda-\zeta_{\ell,m}),
    \qquad
    c_\ell=\frac12
       \min_{k:\,b_{\ell,k}\neq0}
       \frac{a_k}{|b_{\ell,k}|}.
\]
Multiplication by $c_\ell$ preserves the prescribed roots while ensuring that
\[
    |c_\ell b_{\ell,k}|\leq \frac{a_k}{2},
    \qquad k=0,\ldots,4.
\]
Consequently, the first diagonal entry determines the norm of each coefficient matrix:
\[
    \|D_k\|_2
    =\max\left\{
        a_k,\max_{1\leq\ell\leq n-1}
        |c_\ell b_{\ell,k}|
      \right\}
    =a_k.
\]
Thus, $D$ has the desired tropical roots, while $4(n-1)$ of its $4n$ eigenvalues have the prescribed moduli and arguments.

Finally, we choose two independent random orthogonal matrices $U$ and $V$ and define the dense matrix polynomial
\[
    P(\lambda)=UD(\lambda)V^T
    =\sum_{k=0}^{4}P_k\lambda^k,
    \qquad P_k=UD_kV^T.
\]
Since $U$ and $V$ are nonsingular, $P$ and $D$ have the same eigenvalues, including their algebraic multiplicities. 
Hence, in exact arithmetic, the transformation preserves both the tropical roots and the prescribed spectrum while producing dense coefficient matrices.

%

Figure~\ref{fig:prescribed-quartic} shows the right and left polynomial backward errors for $n=50$ and $R=500$. 
The dotted vertical lines mark the tropical roots, and the dashed horizontal line marks unit roundoff.

\begin{figure}[htbp]
    \centering
    \begin{subfigure}{\textwidth}
        \centering
        \includegraphics[width=\textwidth]{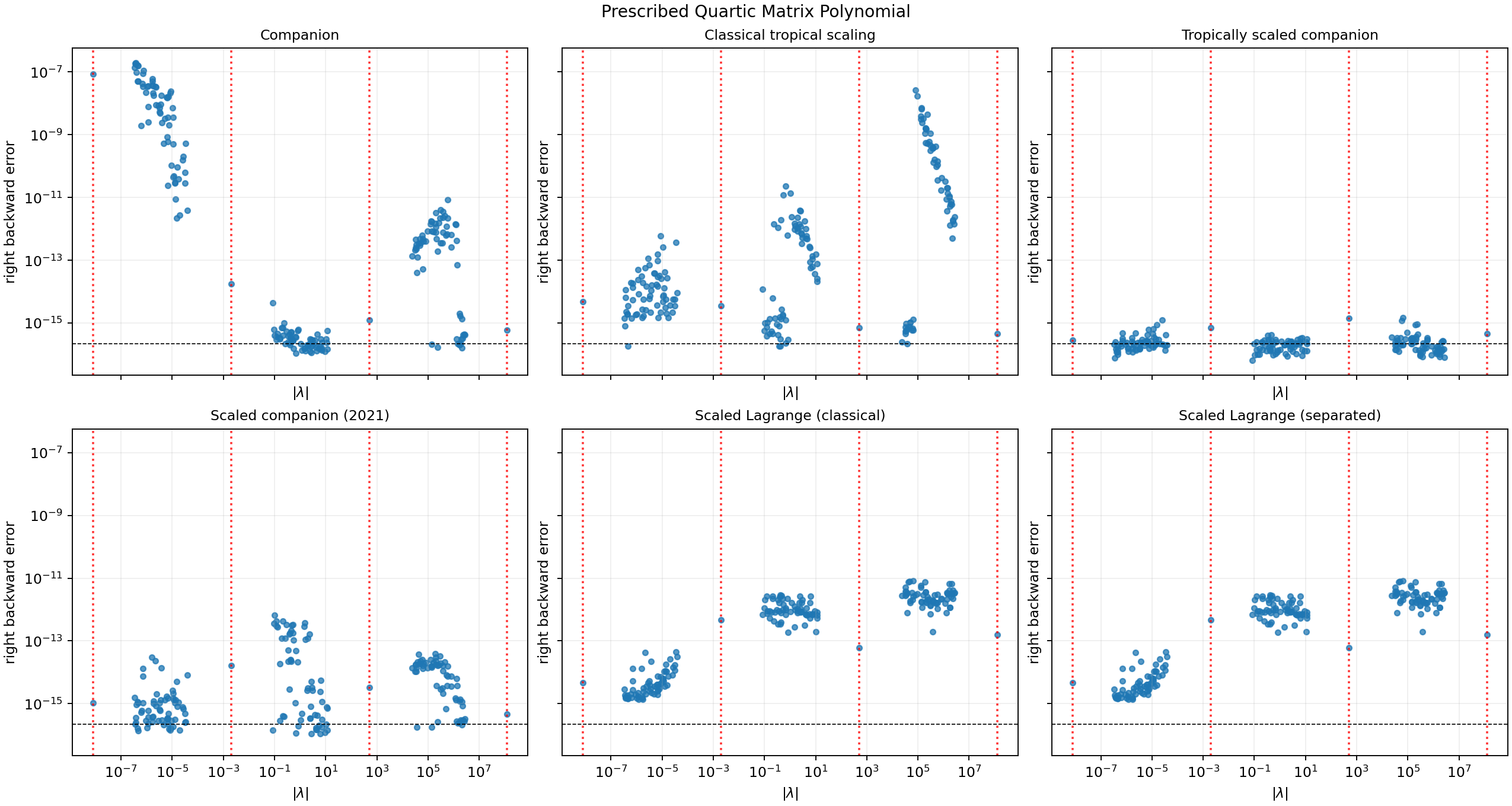}
        \caption{Right polynomial backward errors.}
        \label{fig:prescribed-quartic-right}
    \end{subfigure}

    \medskip

    \begin{subfigure}{\textwidth}
        \centering
        \includegraphics[width=\textwidth]{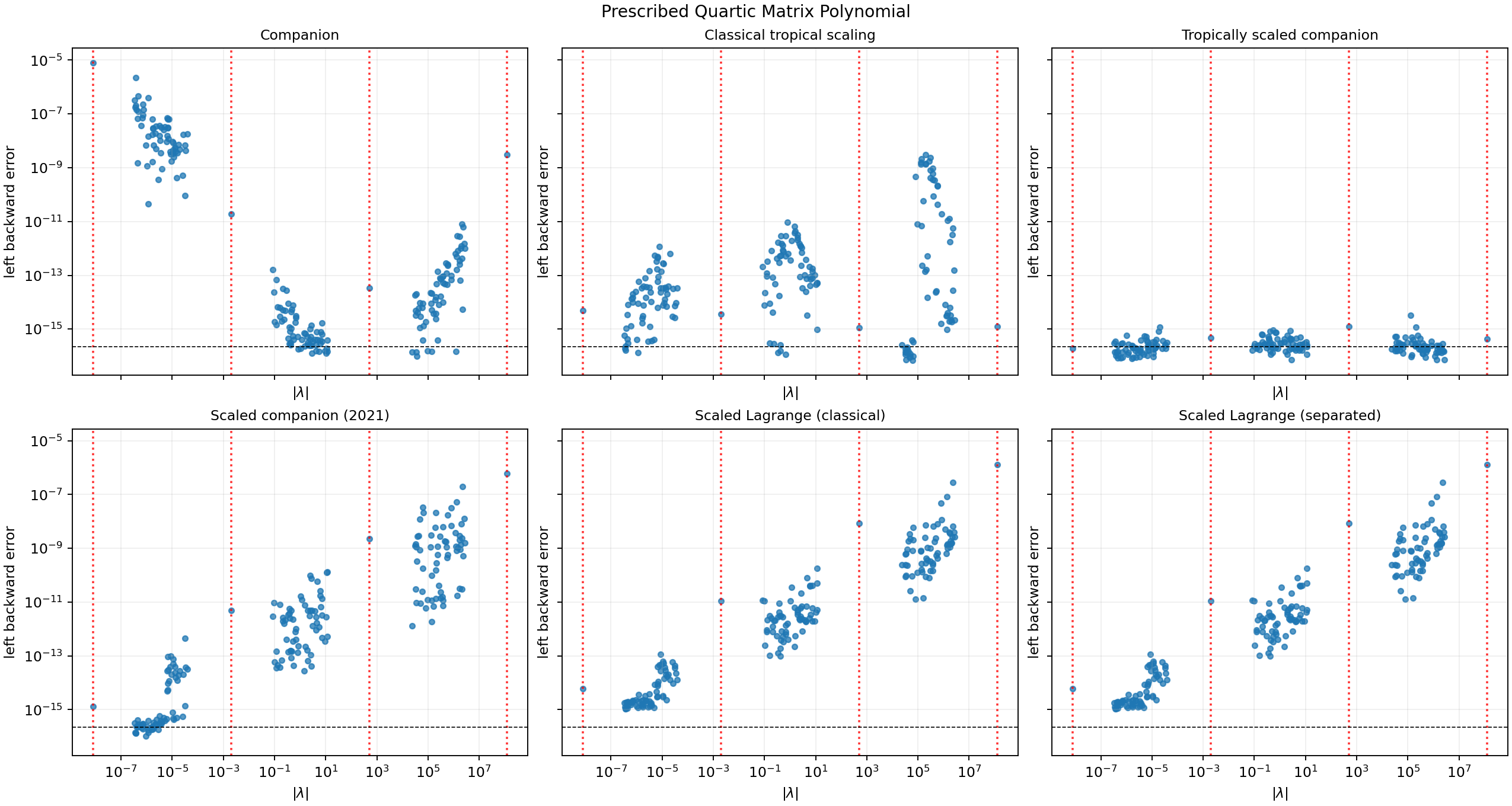}
        \caption{Left polynomial backward errors.}
        \label{fig:prescribed-quartic-left}
    \end{subfigure}

    \caption{Polynomial backward errors for the
    prescribed-spectrum quartic with $n=50$ and $R=500$.
    Dotted vertical lines mark the tropical roots;
    the dashed horizontal line marks unit roundoff.}
    \label{fig:prescribed-quartic}
\end{figure}

For the companion pencil, the backward errors vary considerably with the eigenvalue modulus, with right and left backward errors reaching approximately $10^{-7}$ and $10^{-5}$, respectively, for the smallest eigenvalues.
Classical tropical scaling substantially improves the results, particularly near the tropical roots. 
However, the backward errors remain relatively large for some eigenvalues between tropical roots, reaching approximately $10^{-7}$ for both right and left eigenpairs in the large-modulus group.

The tropically scaled companion linearization of Van Barel and Tisseur and their two Lagrange linearizations give considerably smaller right backward errors, generally between $10^{-15}$ and $10^{-11}$. 
Their left backward errors are less satisfactory, increasing to approximately $10^{-6}$ for the largest eigenvalues. 
The two Lagrange variants give nearly identical results in this example.

In contrast, the proposed family of scaled companion linearizations gives small and nearly uniform backward errors for both right and left eigenpairs.
The backward errors remain close to unit roundoff throughout the recovered spectrum, with most values between $10^{-16}$ and $10^{-15}$.

The companion gives strongly varying accuracy.
Its errors reach about $10^{-3}$ for the smallest eigenvalues, while many larger eigenvalues have much smaller errors.
Tropical scaling substantially improves the results, but no individual scaling works equally well across the spectrum. Even after rank selection, the combined tropical errors reach about $10^{-8}$ in the large-modulus group.

The combined covering gives more uniform accuracy.
The plotted errors remain approximately between $10^{-16}$ and $10^{-15}$ across the recovered spectrum.
This example illustrates the benefit of assigning linearizations to whole radial regions when eigenvalues lie far from the tropical roots.


\subsection{A random quartic with ill-conditioned coefficients}

We next construct a random quartic
\[
    P(\lambda)=\sum_{k=0}^{4}P_k\lambda^k
\]
with prescribed coefficient norms and condition numbers.
For each $k$, choose independent random orthogonal matrices $U_k,V_k$ and set
\[
    P_k=a_kU_k\Sigma_kV_k^T,
    \qquad
    a_k=R^{-k(k-1)/2}.
\]
For the endpoint coefficients, take
$\Sigma_0=\Sigma_4=I_n$.
For $k=1,2,3$, use
\[
    \Sigma_k
    =\operatorname{diag}
      \left(K^{-(m-1)/(n-1)}\right)_{m=1}^{n}.
\]
Thus, every coefficient is invertible,
\[
    \|P_k\|_2=a_k,
    \qquad
    \kappa_2(P_0)=\kappa_2(P_4)=1,
    \qquad
    \kappa_2(P_k)=K\quad(k=1,2,3).
\]

The tropical roots are
\[
    1,\quad R,\quad R^2,\quad R^3,
\]
all with multiplicity one.
We use $n=50$, $R=10^5$, and $K=10^8$.
The coefficient norms then range from $1$ to $10^{-30}$, and the tropical roots range from $1$ to $10^{15}$.

Figure~\ref{fig:random-quartic} compares the right polynomial backward errors.
Red dashed lines mark the tropical roots, and black dotted lines mark the midpoint scales.
The gray horizontal line marks unit roundoff.

\begin{figure}[htbp]
    \centering
    \begin{subfigure}{\textwidth}
        \centering
        \includegraphics[width=\textwidth]{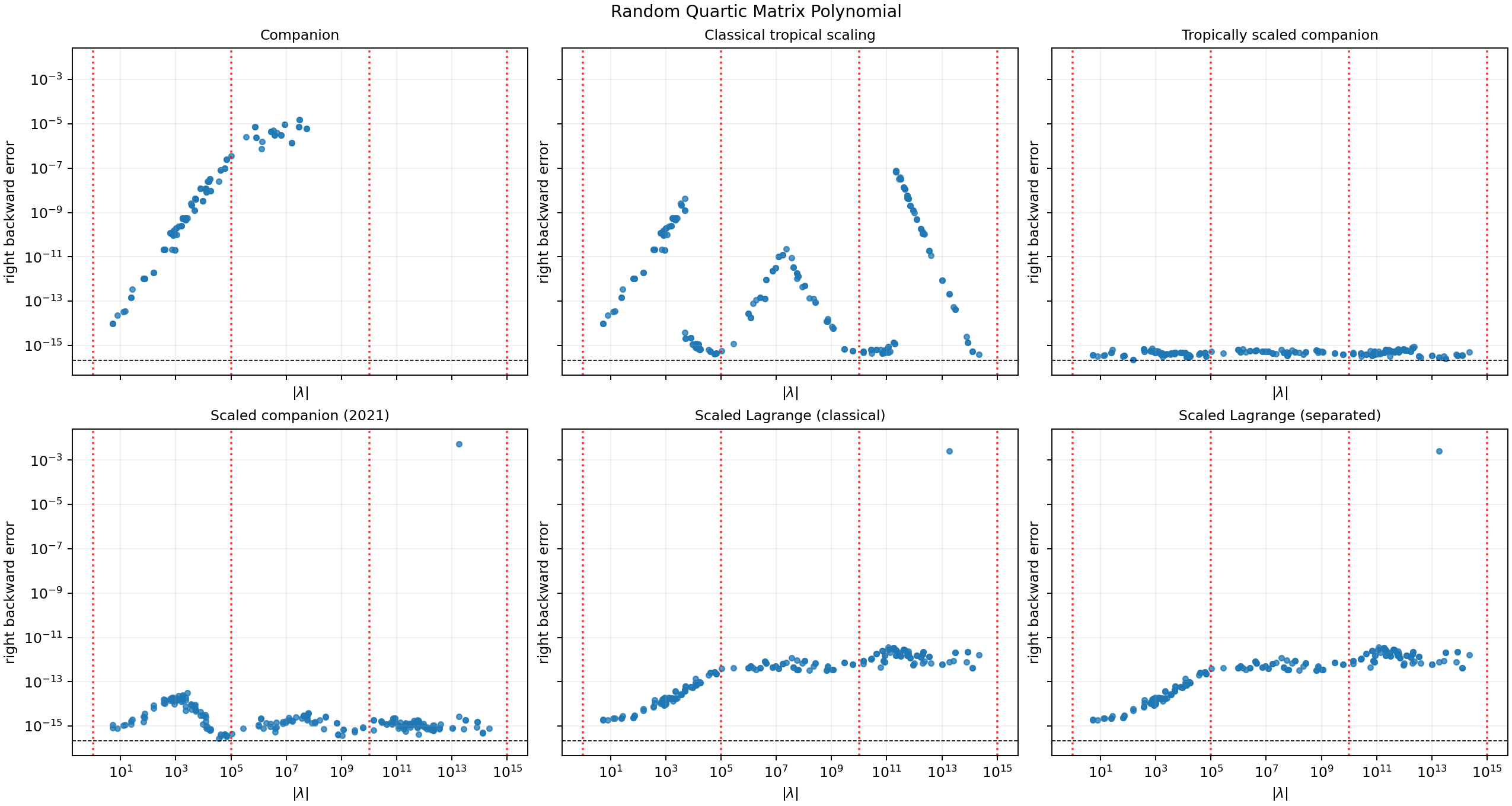}
        \caption{Right polynomial backward errors.}
        \label{fig:random-quartic-right}
    \end{subfigure}

    \medskip

    \begin{subfigure}{\textwidth}
        \centering
        \includegraphics[width=\textwidth]{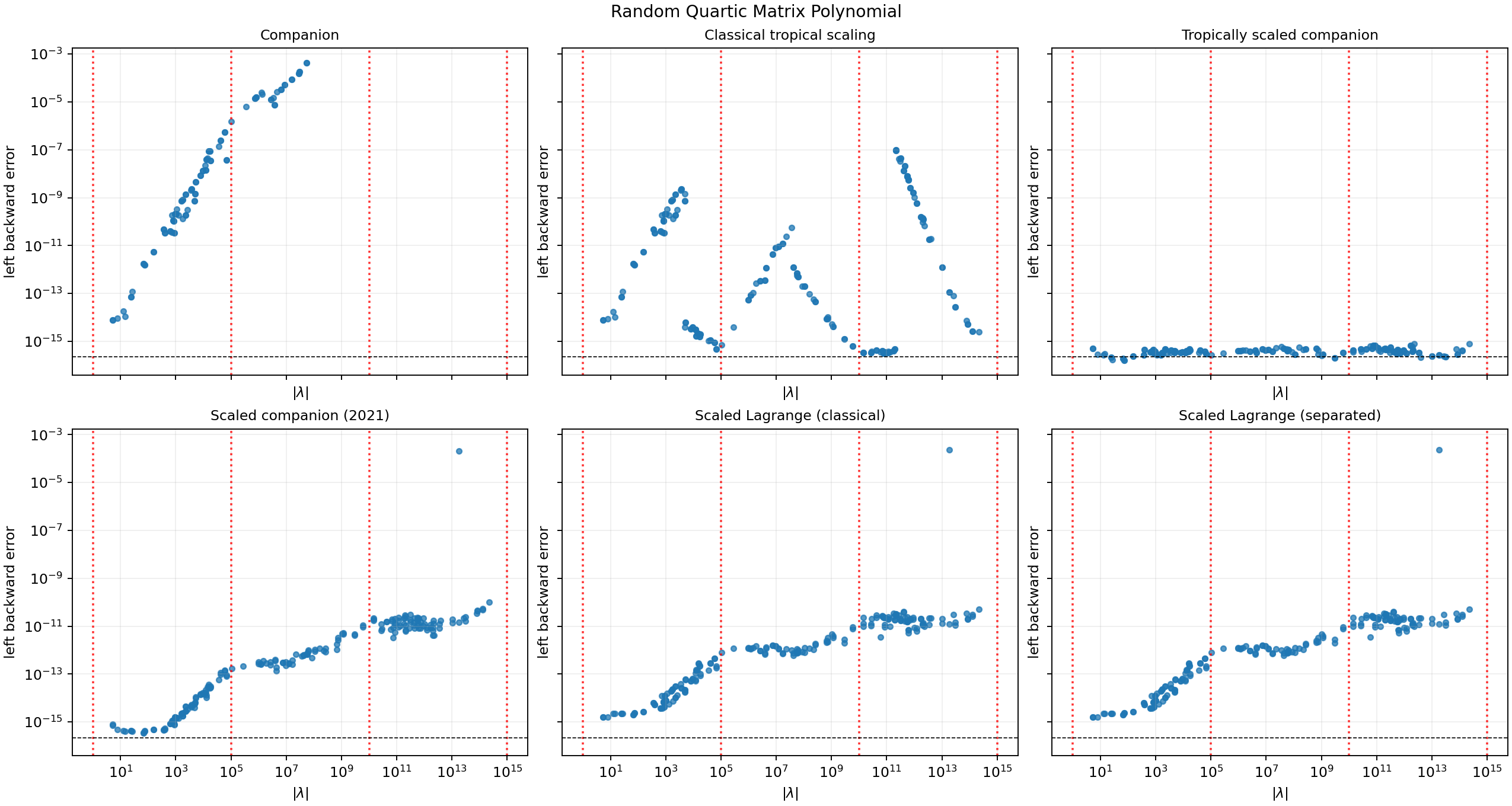}
        \caption{Left polynomial backward errors.}
        \label{fig:random-quartic-left}
    \end{subfigure}

    \caption{Polynomial backward errors for the
    random quartic with $n=50$, $R=10^5$, and $K=10^8$.
    Red dashed lines mark tropical roots;
    black dotted lines mark midpoint scales.
    The gray horizontal line marks unit roundoff.}
    \label{fig:random-quartic}
\end{figure}

The companion pencil performs poorly in this example. 
The QZ algorithm returns a substantial number of eigenvalues as infinite, even though all eigenvalues of the polynomial are finite. 
For the recovered finite eigenpairs, the right backward errors reach approximately $10^{-5}$, while the left backward errors exceed $10^{-4}$.

Classical tropical scaling substantially reduces the backward errors, but they still vary considerably with the eigenvalue modulus.
Both right and left backward errors are small near the tropical roots but increase to approximately $10^{-7}$ for some eigenvalues far from them.

The scaled companion linearization of Tisseur and Van Barel (2021) gives small right backward errors for most eigenpairs, generally between $10^{-16}$ and $10^{-14}$. 
The two scaled Lagrange linearizations give nearly identical results, with right backward errors reaching approximately $10^{-12}$. 
For all three linearizations, the left backward errors tend to increase with the eigenvalue modulus.

In contrast, the proposed family of scaled companion linearizations maintains both right and left backward errors between approximately $10^{-16}$ and $10^{-15}$ throughout the recovered spectrum.


\subsection{The planar waveguide problem}

We next consider the \texttt{planar\_waveguide} quartic matrix polynomial from the NLEVP collection. 
Its eigenvalue moduli span several orders of magnitude, with many eigenvalues lying between the two tropical roots in modulus.

Figure~\ref{fig:planar-waveguide} compares the right
polynomial backward errors.
Red dashed lines mark the tropical roots, and black dotted lines mark the midpoint scales.
The gray horizontal line marks unit roundoff.

\begin{figure}[t]
    \centering
    \begin{subfigure}{\textwidth}
        \centering
        \includegraphics[width=\textwidth]{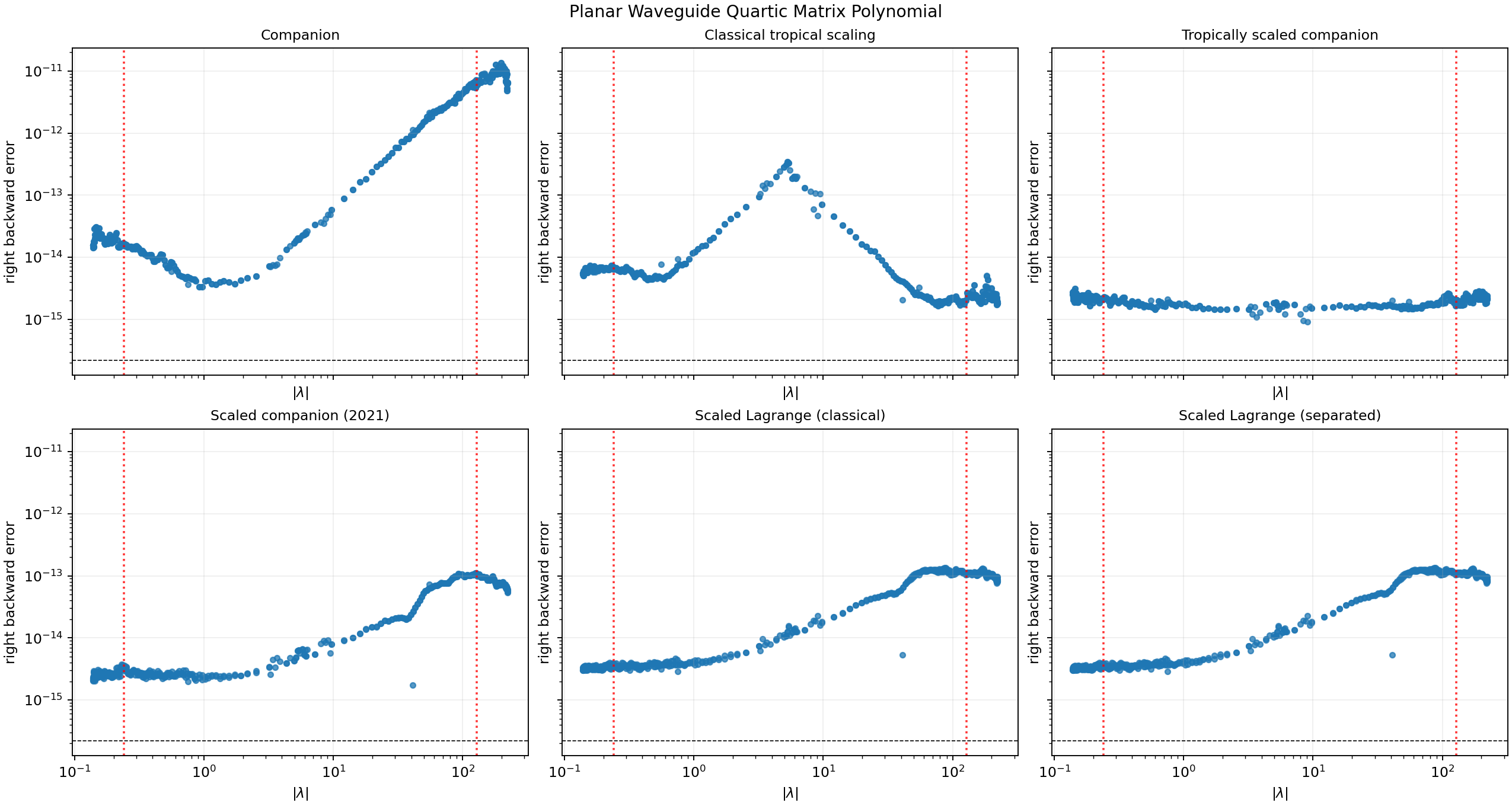}
        \caption{Right polynomial backward errors.}
        \label{fig:planar-waveguide-right}
    \end{subfigure}

    \medskip

    \begin{subfigure}{\textwidth}
        \centering
        \includegraphics[width=\textwidth]{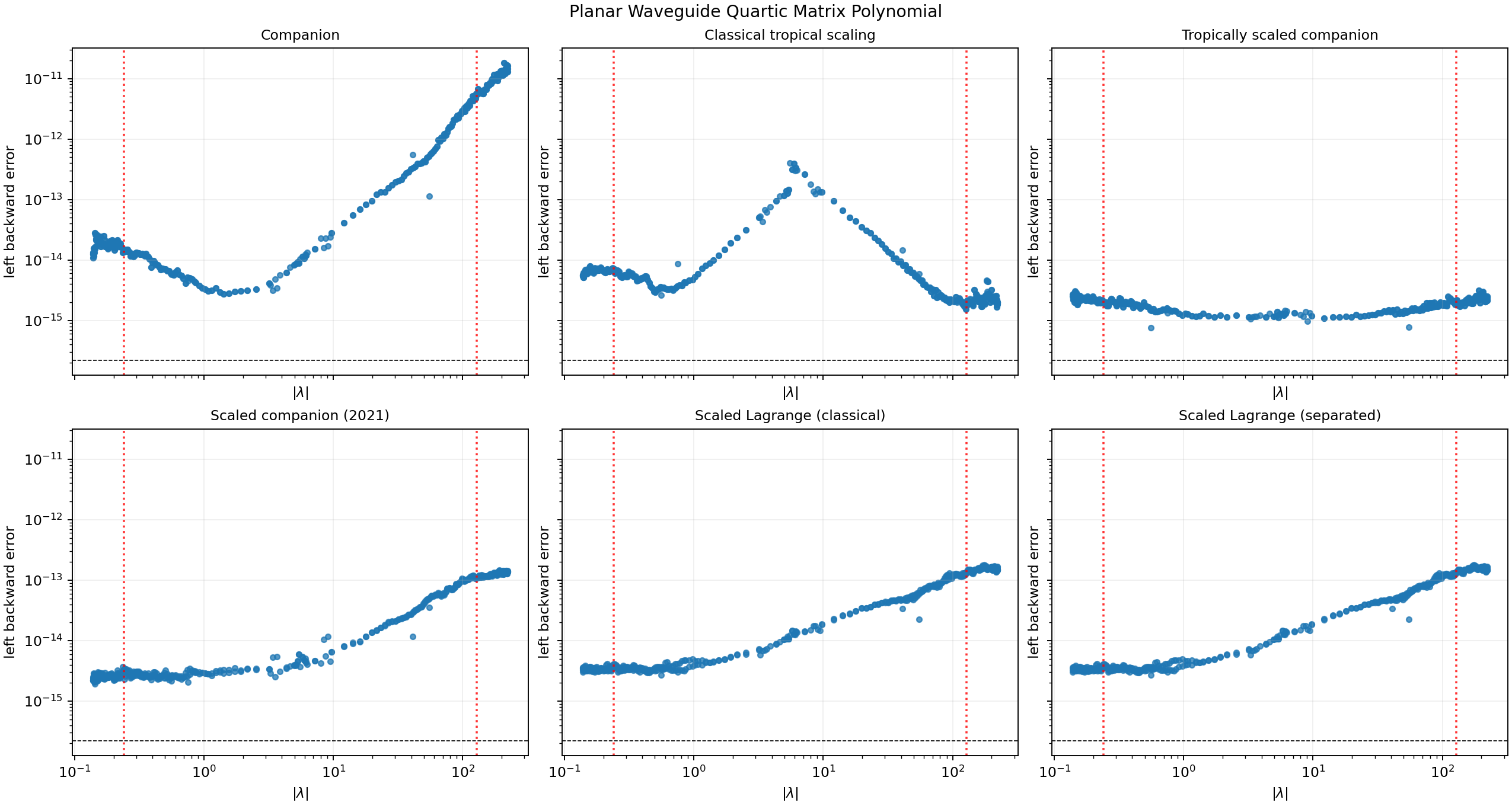}
        \caption{Left polynomial backward errors.}
        \label{fig:planar-waveguide-left}
    \end{subfigure}

    \caption{Coefficientwise relative polynomial backward errors
    for the planar waveguide problem.
    The columns show the unscaled companion pencil,
    classical tropical scaling, and the proposed
    tropically scaled companion pencils.
    Vertical dotted lines mark the tropical roots;
    the horizontal dashed line marks unit roundoff.}
    \label{fig:planar-waveguide}
\end{figure}

For the companion pencil, both right and left backward errors increase substantially with the eigenvalue modulus beyond $|\lambda|\approx 2$, reaching approximately $10^{-11}$ for the largest eigenvalues.

Classical tropical scaling reduces the backward errors near the tropical roots, but the errors increase to approximately $10^{-13}$ for eigenvalues lying between them.

The three scaled linearizations of Tisseur and Van Barel give similar results for both right and left eigenpairs. 
Their backward errors generally increase with the eigenvalue modulus, ranging from approximately $10^{-15}$ to $10^{-13}$. As in the previous examples, the two scaled Lagrange linearizations give nearly identical results.

In contrast, the proposed family of scaled companion linearizations yields both right and left backward errors of order $10^{-15}$ throughout the recovered spectrum. 
Although the differences among the methods are less pronounced than in the previous examples, the proposed approach again yields the smallest and most uniform backward errors.

\bibliographystyle{siamplain}
\bibliography{references}

@article{Tisseur2000,
  author = {Tisseur, Fran{\c{c}}oise},
  title = {Backward error and condition of polynomial eigenvalue problems},
  journal = {Linear Algebra and its Applications},
  volume = {309},
  number = {1--3},
  pages = {339--361},
  year = {2000},
  doi = {10.1016/S0024-3795(99)00063-4}
}

@article{HighamLiTisseur2007,
  author = {Higham, Nicholas J. and Li, Ren-Cang and Tisseur, Fran{\c{c}}oise},
  title = {Backward error of polynomial eigenproblems solved by linearization},
  journal = {SIAM Journal on Matrix Analysis and Applications},
  volume = {29},
  number = {4},
  pages = {1218--1241},
  year = {2007},
  doi = {10.1137/060663738}
}

@article{MackeyEtAl2006,
  author = {Mackey, D. Steven and Mackey, Niloufer and Mehl, Christian and Mehrmann, Volker},
  title = {Vector spaces of linearizations for matrix polynomials},
  journal = {SIAM Journal on Matrix Analysis and Applications},
  volume = {28},
  number = {4},
  pages = {971--1004},
  year = {2006},
  doi = {10.1137/050628350}
}

@article{DopicoEtAl2018,
  author = {Dopico, Froil{\'a}n M. and Lawrence, Piers W. and P{\'e}rez, Javier and {Van Dooren}, Paul},
  title = {Block {Kronecker} linearizations of matrix polynomials and their backward errors},
  journal = {Numerische Mathematik},
  volume = {140},
  number = {2},
  pages = {373--426},
  year = {2018},
  doi = {10.1007/s00211-018-0969-z}
}

@article{MolerStewart1973,
  author = {Moler, Cleve B. and Stewart, G. W.},
  title = {An algorithm for generalized matrix eigenvalue problems},
  journal = {SIAM Journal on Numerical Analysis},
  volume = {10},
  number = {2},
  pages = {241--256},
  year = {1973},
  doi = {10.1137/0710024}
}

@incollection{GaubertSharify2009,
  author = {Gaubert, St{\'e}phane and Sharify, Meisam},
  title = {Tropical scaling of polynomial matrices},
  booktitle = {Positive Systems},
  series = {Lecture Notes in Control and Information Sciences},
  volume = {389},
  publisher = {Springer},
  pages = {291--303},
  year = {2009},
  doi = {10.1007/978-3-642-02894-6_28}
}

@article{NoferiniSharifyTisseur2015,
  author = {Noferini, Vanni and Sharify, Meisam and Tisseur, Fran{\c{c}}oise},
  title = {Tropical roots as approximations to eigenvalues of matrix polynomials},
  journal = {SIAM Journal on Matrix Analysis and Applications},
  volume = {36},
  number = {1},
  pages = {138--157},
  year = {2015},
  doi = {10.1137/14096637X}
}

@article{VanBarelTisseur2018,
  author = {{Van Barel}, Marc and Tisseur, Fran{\c{c}}oise},
  title = {Polynomial eigenvalue solver based on tropically scaled {Lagrange} linearization},
  journal = {Linear Algebra and its Applications},
  volume = {542},
  pages = {182--208},
  year = {2018},
  doi = {10.1016/j.laa.2017.04.025}
}

@article{HammarlingMunroTisseur2013,
  author = {Hammarling, Sven and Munro, Craig J. and Tisseur, Fran{\c{c}}oise},
  title = {An algorithm for the complete solution of quadratic eigenvalue problems},
  journal = {ACM Transactions on Mathematical Software},
  volume = {39},
  number = {3},
  pages = {18:1--18:19},
  year = {2013},
  doi = {10.1145/2450153.2450156}
}

@article{ZengSu2014,
  author = {Zeng, L. and Su, Y.},
  title = {A backward stable algorithm for quadratic eigenvalue problems},
  journal = {SIAM Journal on Matrix Analysis and Applications},
  volume = {35},
  number = {2},
  pages = {499--516},
  year = {2014},
  doi = {10.1137/130921234}
}

@article{DeTeranEtAl2010,
  author = {{De Ter{\'a}n}, Fernando and Dopico, Froil{\'a}n M. and Mackey, D. Steven},
  title = {{Fiedler} companion linearizations and the recovery of minimal indices},
  journal = {SIAM Journal on Matrix Analysis and Applications},
  volume = {31},
  number = {4},
  pages = {2181--2204},
  year = {2010},
  doi = {10.1137/090772927}
}

@article{NakatsukasaEtAl2017,
  author = {Nakatsukasa, Yuji and Noferini, Vanni and Townsend, Alex},
  title = {Vector spaces of linearizations for matrix polynomials: A bivariate polynomial approach},
  journal = {SIAM Journal on Matrix Analysis and Applications},
  volume = {38},
  number = {1},
  pages = {1--29},
  year = {2017},
  doi = {10.1137/15M1013286}
}

@article{BuenoEtAl2018,
  author = {Bueno, Mar{\'i}a I. and Dopico, Froil{\'a}n M. and P{\'e}rez, Javier and Saavedra, Rafael and Zykoski, Brian},
  title = {A simplified approach to {Fiedler}-like pencils via block minimal bases pencils},
  journal = {Linear Algebra and its Applications},
  volume = {547},
  pages = {45--104},
  year = {2018},
  doi = {10.1016/j.laa.2018.01.033}
}

@article{DopicoPerezVanDooren2019,
  author = {Dopico, Froil{\'a}n M. and P{\'e}rez, Javier and {Van Dooren}, Paul},
  title = {Block minimal bases {$\ell$}-ifications of matrix polynomials},
  journal = {Linear Algebra and its Applications},
  volume = {562},
  pages = {163--204},
  year = {2019},
  doi = {10.1016/j.laa.2018.10.010}
}

@article{BetckeEtAl2013,
  author  = {Betcke, Timo and Higham, Nicholas J.
             and Mehrmann, Volker and Schr{\"o}der, Christian
             and Tisseur, Fran{\c{c}}oise},
  title   = {{NLEVP}: A Collection of Nonlinear Eigenvalue Problems},
  journal = {ACM Transactions on Mathematical Software},
  volume  = {39},
  number  = {2},
  pages   = {7:1--7:28},
  year    = {2013},
  doi     = {10.1145/2427023.2427024}
}

@article{TisseurVanBarel2021,
  author  = {Tisseur, Fran{\c{c}}oise and Van Barel, Marc},
  title   = {Min-max elementwise backward error for roots of polynomials
             and a corresponding backward stable root finder},
  journal = {Linear Algebra and its Applications},
  volume  = {623},
  year    = {2021},
  pages   = {454--477},
  doi     = {10.1016/j.laa.2020.10.034}
}

@unpublished{CoefficientwiseGradedPerturbations,
  author = {P\'erez, Javier},
  title  = {Coefficientwise Backward Errors under Graded Perturbations
            of {L}agrange and Companion Linearizations},
  note   = {Manuscript in preparation}
}
\end{document}